\documentclass[10pt]{amsart}

\usepackage[mathcal]{euscript}
\usepackage{amssymb, amsfonts, amsmath}
\usepackage{xypic}
\usepackage{amscd}
\usepackage{dutchcal}
\usepackage{amsthm}
\usepackage[margin=1.5in]{geometry}
\usepackage{xcolor}
\usepackage{comment}

\newtheorem{theorem}{Theorem}
\newtheorem{lemma}[theorem]{Lemma}
\newtheorem{corollary}[theorem]{Corollary}

\newtheorem{theorem-definition}[theorem]{Theorem-Definition}
\newtheorem{proposition}[theorem]{Proposition}

\theoremstyle{definition}
\newtheorem{definition}[theorem]{Definition}

\newtheorem{remark}[theorem]{Remark}
\newtheorem*{theorem*}{Theorem}

\numberwithin{equation}{section} \numberwithin{figure}{section}

\numberwithin{equation}{section}

\newcommand{\F}{\mathbb{F}}

\newcommand{\E}{\E_{\infty}}

\usepackage{graphicx}

\author{Manuel Rivera, Felix Wierstra, Mahmoud Zeinalian}

\newcommand{\Addresses}{{
  \bigskip
  \footnotesize

    \textsc{Manuel Rivera, Department of Mathematics, Purdue University, 150 N. University Street, West Lafayette, IN 47907-2067}\par\nopagebreak \textit{E-mail address} \texttt{manuelr@purdue.edu}

    \medskip

  \textsc{Felix Wierstra, Korteweg-de Vries Institute for Mathematics, University of Amsterdam, Science Park 107, Postbus 94248,
1090 GE Amsterdam, The Netherlands}\par\nopagebreak
  \textit{E-mail address} \texttt{f.p.wierstra@uva.nl}

  \medskip
  \medskip
  
  \textsc{Mahmoud Zeinalian, Department of Mathematics, City University of New York, Lehman College, 250 Bedford Park Blvd W, Bronx, NY 10468
   }\par\nopagebreak
  \textit{E-mail address} \texttt{mahmoud.zeinalian@lehman.cuny.edu}

}
}


\begin{document}

\title[]{The simplicial coalgebra of chains determines homotopy types rationally and one prime at a time}
\maketitle

\begin{abstract}
We prove that the simplicial cocommutative coalgebra of singular chains on a connected topological space determines the homotopy type rationally and one prime at a time, without imposing any restriction on the fundamental group. In particular, the fundamental group and the homology groups with coefficients in arbitrary local systems of vector spaces are completely determined by the natural algebraic structure of the chains. The algebraic structure is presented as the class of the simplicial cocommutative coalgebra of chains under a notion of weak equivalence induced by a functor from coalgebras to algebras coined by Adams as the cobar construction. The fundamental group is determined by a quadratic equation on the zeroth homology of the cobar construction of the normalized chains which involves Steenrod's chain homotopies for cocommutativity of the coproduct. The homology groups with local coefficients are modeled by an algebraic analog of the universal cover which is invariant under our notion of weak equivalence. We conjecture that the integral homotopy type is also determined by the simplicial coalgebra of integral chains, which we prove when the universal cover is of finite type.
 \end{abstract} 

\Addresses

\section{Introduction}

One of the main goals of algebraic topology is to classify topological spaces, up to a specified notion of equivalence, by means of algebraic invariants. In this paper, we use the  singular chains on a space together with the coproduct induced by the diagonal map to classify homotopy types over a field. By combining Adams' work on the cobar construction in \cite{A56} with Steenrod's celebrated work on cohomology operations introduced in \cite{St47}, we use the (homotopy) cocommutativity of the diagonal to recover the fundamental group in its full generality. We further show that the simplicial cocommutative coalgebra of chains determines all homology groups with coefficients in any possible local system of vector spaces over a field. When we assume that the universal cover is of finite type, i.e. all its homology groups are finitely generated, then we also show that the integral homotopy type is completely determined by the simplicial cocommutative coalgebra of chains. We conjecture that this result holds for all connected spaces without any finite type assumptions on the universal cover.

Steenrod operations began a revolution of progress in algebraic topology, which became the rich and successful field of stable homotopy theory. In the unstable setting, the work of Sullivan and Quillen treated rational homotopy types, with strong conditions on the fundamental group, through chain and cochain level algebraic structure. Motivated by the geometric problem of understanding the diffeomorphism class of compact smooth manifolds, Sullivan proved that two simply connected spaces of finite rational type are rationally homotopy equivalent if and only if their commutative differential graded (cdg) algebras of rational polynomial forms are quasi-isomorphic \cite{S77}. Quillen obtained a similar statement for simply connected spaces (without finiteness assumptions) via cocommutative dg rational coalgebras \cite{Q69}. Their results and machinery can be improved to include nilpotent spaces, i.e. spaces with nilpotent fundamental group acting nilpotently on the higher homotopy groups. 

Since the appearance of the results of Sullivan and Quillen, there have been different approaches to classifying spaces up to Bousfield localization or completion with respect to fields of arbitrary characteristic \cite{G95}, \cite{M01} and with respect to integer coefficients \cite{M06}. In some form or another, the cocommutative diagonal map studied by Steenrod, either before or after chain approximation, appears again in all of these works. The end goal of this line of research is to understand in complete generality what a homotopy type is in terms of algebraic data. 

However, all of these approaches involve notions of equivalence which are not strong enough to capture all of the fundamental group. Consequently, many of the statements either require strong restrictions on the fundamental group or determine spaces up to ambiguity on the fundamental group. For example, Goerss showed in \cite{G95} that the simplicial cocommutative coalgebra of chains over a field determines spaces up to Bousfield localization, a notion of localization for spaces under which the fundamental group is not preserved. In our approach, we follow a divide and conquer strategy by first obtaining the fundamental group from the algebraic structure of the chains, constructing the universal cover, and then applying localization techniques to the universal cover taking advantage of its simple connectivity. The main result of this article is the following.
\\
\\
\noindent \textbf{Main Theorem.} \textit{For any field $\mathbb{F}$, two reduced Kan complexes $X$ and $Y$ can be connected by a zig-zag of $\pi_1$-$\mathbb{F}$-equivalences if and only if their connected simplicial cocommutative coalgebras of chains $\mathbb{F}X$ and $\mathbb{F}Y$ can be connected by a zig-zag of $\Omega$-quasi-isomorphisms.}
\\

We briefly explain the terminology in the above statement, its significance, and the main ingredients used in the proof. A Kan complex is \textit{reduced} if it has a single vertex. For example, any pointed topological space $(Z,z)$ gives rise to a reduced Kan complex $\text{Sing}(Z,z)$ whose $n$-simplices consist of all continuous maps $\sigma: \Delta^n \to Z$ such that $\sigma(v_i)=z$ for all vertices $v_0,...,v_n \in \Delta^n$.

Let $R$ be an arbitrary commutative unital ring. A map $f:X\to Y$ between reduced Kan complexes is a \textit{$\pi_1$-$R$-equivalence} if it induces an isomorphism of fundamental groups $$\pi_1(f): \pi_1(X) \xrightarrow{\cong} \pi_1(Y)$$ and an isomorphism $$H_*(\widetilde{f};R): H_*(\widetilde{X};R) \xrightarrow{\cong} H_*(\widetilde{Y};R)$$ between the homology groups with $R$-coefficients of the universal covers. Equivalently, a map $f$ is a $\pi_1$-$R$-equivalence if and only if it induces an isomorphism on fundamental groups and on all homology groups with values in every possible local system of $R$-modules. If $f: X \to Y$ is a $\pi_1$-$R$-equivalence then $f$ is an $R$-homology equivalence, i.e. $H_*(f;R): H_*(X;R) \to H_*(Y;R)$ is an isomorphism, but not vice-versa. Note that a map between reduced Kan complexes is a $\pi_1$-$\mathbb{Z}$-equivalence if and only if it is a homotopy equivalence.

A \textit{simplicial cocommutative $R$-coalgebra} $C$ is a simplicial object in the category of cocommutative counital $R$-coalgebras. To every simplicial set $X$ we can associate a simplicial cocommutative coalgebra by defining $RX$ as the free $R$-module generated by the simplices of $X$ and with the face and degeneracy maps induced by the face and degeneracy maps of $X$. The cocommutative coalgebra structure is the one induced by the diagonal map of simplicial sets $X \rightarrow X \times X$. It further turns out that if the simplicial set $X$ is reduced, then $RX$ is coaugmented and connected, meaning that it is one dimensional in degree $0$ and that there is a canonical map from $R$ to $RX$, where $R$ is seen as the constant simplicial cocommutative coalgebra.
 
To each simpicial cocommutative  coalgebra $C$, we can functorially associate a differential graded coassociative coalgebra $N_*(C)$ which is called the normalized chains. When $C$ is connected (resp. coaugmented) then $N_*(C)$ is connected (resp. coaugemented) as well. We say that a morphism $f:C \rightarrow D$ of connected  simplicial cocommutative coalgebras is an $\Omega$\textit{-quasi-isomorphism} if the induced morphism of normalized chains is a quasi-isomorphism after applying the cobar construction $\Omega$, i.e. if the map
\[
\Omega N_*(f): \Omega N_*(C) \rightarrow \Omega N_*(D),
\] 
 is a quasi-isomoprhism. Any $\Omega$-quasi-isomorphism is a quasi-isomorphism but not vice-versa.

The proof of our main theorem relies on the following constructions and results, which hold over an arbitrary integral domain $R$ and are also of independent interest:
\\
\indent 1) To any connected simplicial cocommutative coalgebra we may associate functorially a \textit{fundamental bialgebra}, a construction which is homotopical in the sense that it is invariant under $\Omega$-quasi-isomorphisms of simplicial cocommutative coalgebras.
\\
\indent 2) The fundamental bialgebra of the simplicial coalgebra of chains $RX$ on any reduced Kan complex $X$ is naturally isomorphic to the fundamental group Hopf algebra $R[\pi_1(X)]$.  In other words, the natural (co)algebraic structure of the chains $RX$ on a reduced Kan complex $X$ determines the fundamental group $\pi_1(X)$ in complete generality, through the \textit{group-like elements functor}, and this data is preserved along $\Omega$-quasi-isomorphisms. More precisely, $\pi_1(X)$ is determined by the quadratic equation 
$$\nabla (\alpha) = \alpha \otimes \alpha$$
where $$\nabla: H_0(\Omega N_*(RX )) \to H_0(\Omega N_*(RX )) \otimes H_0(\Omega N_*(RX ))$$ is a coproduct on the zeroth-homology of the cobar construction induced by the $E_2$-coalgebra structure of the normalized chains $N_*(RX)$. The coproduct $\nabla$ is therefore part of the higher hierarchy of homotopies introduced by Steenrod in \cite{St47} and described in terms of the $E_{\infty}$-operadic framework in \cite{BF04}. The extension of Adams' classical cobar theorem to non-simply connected spaces, proven by the first and third author, lies at the bottom of the fact that the fundamental group can be determined algebraically from the chains \cite{RZ16}, \cite{R19}.  
\\
\indent 3) To any connected simplicial cocommutative coalgebra we may associate functorially a \textit{universal cover}, which is a new simplicial cocommutative coalgebra equipped with an action of the fundamental bialgebra. This construction mirrors the passage from a pointed space to its universal cover. The main idea is equipping a simplicial version of Brown's twisted tensor product with an appropriate non-linear algebraic structure. These constructions, which constitute the key technical input developed in the paper, are introduced as part of a more general theory of simplicial twisted tensor products and simplicial twisting cochains. 
\\
\indent 4) We apply Goerss' results from \cite{G95} relating simplicial cocommutative coalgebras over a field to Bousfield localization at the level of universal covers.

Our main result is particularly important in the cases when $\mathbb{F}= \mathbb{Q}$, the field of rational numbers, and when $\mathbb{F}=\mathbb{F}_p$, the finite field with $p$-elements for a prime $p$. Recall that if a map of spaces induces an isomorphism on homology with $\mathbb{Q}$-coefficients and an isomorphism on homology with $\mathbb{F}_p$-coefficients for each prime $p$, then it induces an isomorphism on integral homology.  This observation, together with the main result of this paper, the classical fracture theorems and arithmetic square, and the fact that $\pi_1$-$\mathbb{Z}$-equivalences are exactly homotopy equivalences, leads us to pose the following conjecture.
\\
\\
\noindent \textbf{Conjecture. }\textit{Two reduced Kan complexes $X$ and $Y$ are homotopy equivalent if and only if their connected simplicial cocommutative coalgebras of chains $\mathbb{Z}X$ and $\mathbb{Z}Y$ can be connected by a zig-zag of  $\Omega$-quasi-isomorphims.}
\\
\\
\indent One direction of the above conjecture already follows from \cite{RWZ18}, where we showed that a map $f: X\to Y$ is a homotopy equivalence between reduced Kan complexes if and only if $\mathbb{Z}f: \mathbb{Z}X \to \mathbb{Z}Y$ is an $\Omega$-quasi-isomorphism, extending a classical theorem of Whitehead. 

The strongest results in the problem of finding complete algebraic models for spaces over fields of arbitrary characteristic have appeared in the work of Mandell \cite{M06}, \cite{M01}.  Mandell proved a classification theorem for nilpotent finite type $p$-complete spaces using the framework of $E_{\infty}$-algebras, an up to (coherent) homotopy version of commutative algebras \cite{M01}. $E_{\infty}$-(co)algebras may be interpreted to be more ``algebraic" than simplicial (co)algebras  in the sense that they are described in terms of operations and relations on an abelian group using the framework of operads and do note involve a ``spatial parameter"  directly as in the case of simplicial coalgebras. Moreover, Mandell describes the sense in which the functor of cochains considered as a $\mathbb{F}_p$-$E_{\infty}$-algebra is homotopically fully faithful on nilpotent finite type $p$-complete spaces. In this theory, $E_{\infty}$-algebras are considered under quasi-isomorphism, a notion suitable to study nilpotent finite type spaces but not strong enough to capture the fundamental group in complete generality.

Mandell goes further and proves an integral detection statement by means of an arithmetic square argument \cite{M06}. Namely, he proves that two nilpotent spaces of finite type $X$ and $Y$ are weak homotopy equivalent if and only if their $E_{\infty}$-algebras of integral singular cochains are quasi-isomorphic. We use this result, together with our constructions, to prove the following special case of the above conjecture.
\\
\\
\noindent \textbf{Theorem. }\textit{Let $X$ and $Y$ be two reduced Kan complexes whose universal covers are of finite type. If the integral chains $\mathbb{Z}X$ and $\mathbb{Z}Y$ can be connected by a zig-zag of $\Omega$-quasi-isomorphisms of connected simplicial cocommutative coalgebras each of which is projective as a $\mathbb{Z}$-module, then $X$ and $Y$ are homotopy equivalent.}
\\

We also conjecture that two connected Kan complexes are homotopy equivalent if and only if their $E_{\infty}$-coalgebras of chains are $\Omega$-quasi-isomorphic. This conjecture could also possibly be improved to arbitrary simplicial sets by incorporating an algebraic localization procedure into the notion of $\Omega$-quasi-isomorphism. Furthermore, there is recent evidence that one may be able to obtain a fully faithful (integral) model for homotopy types  by considering more algebraic structure. This has been studied in a recent preprint of Yuan, \cite{Y19}, where the extra structure is encoded using spectra but not quite in terms of the philosophy and algebraic structures considered by Mandell, who encodes spaces in terms of discrete abelian groups equipped with a countable number of operations and relations considered under a ``good" homotopical notion of weak equivalence.

The general goal of this program is to understand in purely algebraic terms what a homotopy type is and then use the resulting algebraic models effectively. This general problem includes several subtleties such as making a mathematically precise formulation of what ``algebraic" means, a point we do not address in this article. The term ``algebraic" may have different interpretations such as computable algebraic, operadic and derived algebraic, positive degree algebraic, or modeled by an abelian category, and so on. For each interpretation of the term one may try to explore to what extent one can model homotopy types. We also believe the results of this general program, including the main theorem of this paper, will be useful in the study of the topology of geometric spaces, such as compact $3$-manifolds, with arbitrary fundamental group.

The organization of this article is as follows. In section 2 we discuss algebraic preliminaries and discuss the notion of $\Omega$-quasi-isomorphism. In section 3 we recall those parts of \cite{G95} that are relevant for this article and discuss the notion of $\pi_1$-$R$-equivalence. In section 4 and section 5 the main technical tools are developed, these include a theory of simplicial twisted tensor products for simplicial coalgebras and simplicial algebras through which we obtain the notion of the universal cover of a connected simplicial cocommutative coalgebra as a special case. In section 6 we prove our main theorem by applying the machinery developed in the previous sections. Finally in section 7 we prove a special case of our conjecture in the integral case.
\\
\\
\noindent \textbf{Acknowledgements.}
The second author was supported by grant number 2019-00536 from the Swedish Research Council. The second and third author would further like to thank the Max Planck Institute for Mathematics for its hospitality and excellent working conditions. The second author would also like to thank the Graduate Center of the City University of New York for their hospitality and excellent working conditions during his stay there. We would like to thank Dennis Sullivan, Martin Bendersky, Omar Antol\'in Camarena, Michael Mandell, Kathryn Hess, Rob Thompson, and George Raptis for stimulating discussions, exchanges, and comments. We would like to thank the anonymous referee for general comments which helped us improve the introduction. 

\section{Algebraic Preliminaries}
In this section we introduce notation, recall a several algebraic definitions and constructions, and discuss the notion of $\Omega$-\textit{quasi-isomorphism} between simplicial coalgebras. This notion was originally proposed in Lefevre-Hasegawa's thesis and it is essential in Koszul duality theory of algebraic structures, see \cite{LH03}, \cite{LV12}. In the Lie context, a similar notion was used in \cite{HS97}. 

\subsection{Algebras and coalgebras}

Let $R$ be a commutative ring with unit. All tensor products will be over $R$ unless stated otherwise. In some of the statements in this paper, we will assume that $R$ is a field; when this is the case we will denote this field by $\mathbb{F}$. Later on, in section 6, we will denote an arbitrary algebraically closed field by $\mathbb{E}$.

We will consider $R$-algebras and $R$-coalgebras in two different settings: the \textit{differential graded} (dg) setting and the \textit{simplicial} setting. For the definitions of dg algebras and dg coalgebras we refer the reader to \cite{LV12}. All differentials will have degree $-1$. 
 
Denote by $\textbf{Alg}_{R}$ and $\textbf{Coalg}_{R}$ the categories of associative unital augmented $R$-algebras and coassociative counital coaugmented $R$-coalgebras, respectively.  Denote by $\textbf{dgAlg}_{R}$ the category of differential graded augmented associative $R$-algebras and by $\textbf{dgCoalg}_{R}$ the category of dg coaugmented conilpotent coassociative $R$-coalgebras. In this paper, all (co)algebras will be (co)associative and (co)unital. We say a (co)algebra is $R$-\textit{flat} if it is flat as an $R$-module. 

We say that $C \in \textbf{dgCoalg}_{R}$ is \textit{connected} if it is non-negatively graded and the coaugmentation $c: R \to C$ induces an isomorphism $R \cong C_0$. Let $\textbf{dgCoalg}^0_{R}$ be the full subcategory of connected dg coalgebras in $\textbf{dgCoalg}_{R}$.

Let $\mathbf{\Delta}$ be the simplex category. A \textit{simplicial algebra} is a functor $A:\mathbf{\Delta}^{op} \to \textbf{Alg}_{R}$ and a \textit{simplicial coalgebra} is a functor $C: \mathbf{\Delta}^{op} \to  \textbf{Coalg}_{R}$. Denote by $\textbf{sAlg}_{R}$ and $\textbf{sCoalg}_{R}$ the categories of simplicial algebras and simplicial coalgebras, respectively, with natural transformations of functors as morphisms. If $C$ is a simplicial coalgebra we write $C([n])=C_n$, so each $C_n$ is equipped with a coassociative coproduct usually denoted by $$\Delta_n: C_n \to C_n \otimes C_n.$$ 
Equivalently, a simplicial (co)algebra is a set of (co)algebras $\{V_0, V_1, V_2,...\}$ equipped with face maps $d^n_i: V_n \to V_{n-1}$, $n>0$, $0 \leq i \geq n$ and degeneracy maps $s^n_j: V_n \to V_{n+1}$, $n \leq0$, $0\leq i \leq n$, which are all (co)algebra maps and satisfy the simplicial identities. 

We say that $C$ is a simplicial \textit{cocommutative} coalgebra if each $(C_n, \Delta_n)$ is cocommutative for all $n \geq 0$. A simplicial coalgebra $C$ is \textit{connected} if there is an isomorphism of coalgebras $(C_0, \Delta_0) \cong R$, where $R$ is given the coproduct determined by $1 \mapsto 1 \otimes 1$. We denote by $\textbf{scCoalg}_{R} \subset \textbf{sCoalg}_{R}$ the full subcategory of simplicial cocommutative coalgebras and by $\textbf{scCoalg}^0_{R} \subset \textbf{sCoalg}_{R}$ the full subcategory of connected simplicial cocommutative coalgebras.

The tensor product of simplicial (co)algebras $V \otimes W$ is defined degree-wise by setting $(V \otimes W)_n= V_n \otimes W_n$ with face and degeneracy maps obtained by tensor product, i.e. $d^{V \otimes W}_i= d^V_i \otimes d^W_i$ and $s^{V \otimes W}_i= s^V_i \otimes s^W_i$. 

Any simplicial coassociative coalgebra gives rise to a dg coassociative coalgebra through the \textit{normalized chains functor} $$N_*: \textbf{sCoalg}_{R} \to \textbf{dgCoalg}_{R}$$ defined as follows. Given a simplicial coassociative coalgebra $C$ with coproducts $\Delta_n: C_n \to C_n \otimes C_n$, let $(N_*(C), \partial)$ be the dg $R$-module obtained as the quotient $N'_*(C)/ D_*(C)$ where $N'_n(C)= C_n$ equipped with differential $$\partial= \sum_i(-1)^{i}d_i: N'_{*}(C) \to N'_{*-1}(C)$$ given by the alternating sum of the face maps of $C$, and $D_*(C) \subset N'_*(C)$ is the sub-complex generated by degenerate elements. The chain complex $(N_*(C), \partial)$ becomes a dg coassociative coalgebra when equipped with the coproduct
$$\delta: N_*(C) \xrightarrow{N_*(\Delta)} N_*(C \otimes C) \xrightarrow{AW} N_*(C) \otimes N_*(C).$$
In the above composition, $AW$ is the Alexander-Whitney map, which is given on any $x \otimes y \in (C \otimes C)_n= C_n \otimes C_n$ by $$AW(x \otimes y)=  \sum_{p+q=n} f_p(x) \otimes l_q(y),$$
where $f_p$ denotes the front $p$-face induced by the map $[p] \to [p+q]$ in $\mathbf{\Delta}$ determined by $i \mapsto i$ and $l_p$ is the last $q$-face induced by the map $[q] \to [p+q]$ in $\mathbf{\Delta}$ determined by $i \mapsto i+p$. The construction $(C, \Delta) \mapsto (N_*(C), \partial, \delta)$ is natural with respect to maps of simplicial coalgebras and consequently defines a functor.

Finally, we briefly recall the notions of bialgebras and Hopf algebras. An \textit{$R$-bialgebra} $B=(B,\mu,\nabla, u, \epsilon)$ consists of an $R$-module $B$ equipped with a unital algebra structure $(\mu: B \otimes B \to B, u: R \to B)$ together with a counital coalgebra structure $(\nabla: B \to B \otimes B, \epsilon: B \to R)$ which are compatible in the sense that the coproduct $\nabla$ and the counit $\epsilon$ are algebra maps. As a consequence of this definition, we have that the product $\mu$ and the unit $u$ are coalgebra maps. A map of bialgebras is a linear map which is simultaneously  an algebra and a coalgebra map. We also have dg and simplicial versions of bialgebras defined analogously to algebras and coalgebras. We will use the notation $\textbf{scBialg}_{R}$ to denote the category of simplicial cocommutative bialgebras. Note that the simplicial bialgebras in $\textbf{scBialg}_{R}$ are not required to be commutative.
 
A bialgebra $B=(B, \mu, \nabla, u, \epsilon)$ is a \textit{Hopf algebra} if there is a map $s: B \to B$, satisfying
$$\mu \circ (s \otimes id) \circ \nabla =u \circ \epsilon= \mu \circ (id \otimes s) \circ \nabla.$$ The map $s: B \to B$ is called the \textit{antipode}. If a bialgebra has an antipode then it is unique. A map of Hopf algebras is a map of underlying bialgebras. Any map of Hopf algebras preserves the antipodes. 

\subsection{Bar and cobar constructions}  We now recall the bar and cobar functors. We refer to \cite{EM53}, \cite{A56} and \cite{HMS74} for further details. 

Let $(A, d_A) \in \textbf{dgAlg}_{R}$ and suppose $(M,d_M)$ and $(N,d_N)$ are right and left dg $A$-modules, respectively. We denote both the $A$-action and the product in $A$ by $a \cdot b$. Recall that the \textit{two sided bar construction} is the chain complex $(B_*(N, A,M), \partial)$ whose underlying graded $R$-module is given by $$B_p(N,A,M):= (N \otimes Ts\overline{A} \otimes M)_p,$$ where $ \overline{A} = \text{ker} (a)$ denotes the kernel of the augmentation $a: A \to R$,  $s$ is the shift by $+1$ functor, and $$Ts\overline{A}= R \oplus s\overline{A} \oplus (s\overline{A})^{\otimes 2} \oplus (s\overline{A})^{\otimes 3} \oplus ... $$ 
The subscript on $(N \otimes Ts\overline{A} \otimes M)_p$ means total degree $p$ elements in $N \otimes Ts\overline{A} \otimes M$. In what follows we will drop the $s$ for notational simplicity. We write tensors in $B_n(N,A,M)$ as
$n[a_1|...|a_k] m$, where $n \in N, m \in M$ and $a_i \in \overline{A}$ for $i=1,...,k$. Hence, $n[a_1|...|a_k] m \in B_p(N,A,M)$ means that $|n| + |a_1|+ ...+|a_k| + k +|m|=p$. The differential $d_{\text{bar}}: B_p(N,A,M) \to B_{p-1} (N,A,M)$ is defined by
\begin{eqnarray*}
d_{\text{bar}}( n[a_1|...|a_k] m) =  d_N(n)[a_1|...|a_k]m + \sum_{i=1}^k (-1)^{\epsilon_i} n[a_1|...|d_Aa_i|...|a_k] m 
\\+ (-1)^{|n| + |a_1| + ...|a_k|+k} n[a_1|...|a_k] d_M(m)
\\+(-1)^{|n|}( n\cdot a_1)[a_2|...|a_k] m + \sum_{i=2}^k(-1)^{\epsilon_i-1}n[a_1|....|(a_{i-1}\cdot a_{i})| ...|a_k]m
\\+(-1)^{|n|+ |a_1| + ...|a_k|+k-1} n[a_1|...|a_{k-1}] (a_k\cdot m),
\end{eqnarray*}
where $\epsilon_i= |n|+ |a_1+ ...+|a_i|+i$. It is straightforward to check that $d_{\text{bar}}^2=0$. In this article, we will only consider the two sided bar construction $B(R, A, M)$ where $M$ is a left dg $A$-module and $R$ is considered as a right dg $A$-module concentrated in degree $0$ and with right $A$-action induced by the augmentation $a: A \to R$. 

We now define the version of the cobar construction which is relevant for this article. The \textit{cobar construction} is a functor 
$$\Omega: \textbf{dgCoalg}^0_{R} \to \textbf{dgAlg}_{R}$$
defined as follows. For any $C=(C, \partial_C, \Delta) \in \textbf{dgCoalg}^0_{R}$, the underlying graded algebra  of $\Omega(C)$ is the tensor algebra $$Ts^{-1}\overline{C}= R \oplus s^{-1}\overline{C} \oplus (s^{-1}\overline{C})^{\otimes 2} \oplus (s^{-1}\overline{C})^{\otimes 3} \oplus ... ,$$ where $\overline{C}:= C/C_0$, and $s^{-1}$ is the shift by $-1$ functor. We denote monomials in $Ts^{-1}\overline{C}$ by $\{x_1 |...|x_k\}$ where $x_i \in \overline{C}$, dropping the $s^{-1}$ for notational simplicity. Hence, the degree of $\{x_1 |...|x_k\} \in Ts^{-1}\overline{C}$ is $|x_1| + ... + |x_k|-k$. The augmentation is given by the canonical projection $a: Ts^{-1}\overline{C} \to R$. The differential is defined by extending the linear map $$- s^{-1} \circ \partial  \circ s^{+1} + (s^{-1} \otimes s^{-1}) \circ \Delta  \circ s^{+1}: s^{-1}\bar{C} \to T(s^{-1} \bar C)$$ as a derivation to obtain a map $D: T(s^{-1} \bar{C}) \to T(s^{-1} \bar{C})$. The coassociativity of $\Delta$, the compatibility of $\partial$ and $\Delta$, and the fact that $\partial^2 =0$ together imply that $D^2=0$.

\subsection{Weak equivalences of coalgebras} One of the goals of this paper is to understand the homotopical meaning of the following two notions of weak equivalences between (simplicial and dg) coalgebras.

 \begin{definition} 
 
(a) A map $f: C \to C'$ in $\textbf{dgCoalg}_{R}$ is a \textit{quasi-isomorphism of dg coalgebras} if the induced map on homology $H_*(f): H_*(C) \to H_*(C')$ is an isomorphism. 
 \\
 (b) A map $f: C \to C'$ in $\textbf{sCoalg}_{R}$ is a \textit{quasi-isomorphism of simplicial coalgebras} if the induced map of dg coalgebras $N_*(f): N_*(C) \to N_*(C')$ after applying the normalized chains functor is a quasi-isomorphism of dg coalgebras.
 \end{definition}
Recall that a map $f:C \to C'$ in $\textbf{sCoalg}_{R}$ is a quasi-isomorphism if and only if $f$ is a weak homotopy equivalence between the underlying simplicial sets of $C$ and $C'$. 

We also have the following notions, which are stronger than the ones defined above.
 
 \begin{definition}
 (a) A map $f: C \to C'$ in $\textbf{dgCoalg}^0_{R}$ is an \textit{$\Omega$-quasi-isomorphism of connected dg coalgebras} if the induced map after applying the cobar functor $\Omega(f): \Omega(C) \to \Omega(C')$ is a quasi-isomorphism of dg algebras.
 \\
 (b) A map $f: C \to C'$ in  $\textbf{sCoalg}^0_{R}$ is an \textit{$\Omega$-quasi-isomorphism of connected simplicial coalgebras} if the induced map of dg coalgebras $N_*(f): N_*(C) \to N_*(C')$ after applying the normalized chains functor is an $\Omega$-quasi-isomorphism of connected dg coalgebras.
 \end{definition}
 
\begin{proposition}\label{cobarstrong}
Any $\Omega$-quasi-isomorphism between connected dg $R$-flat coalgebras is a quasi-isomor\-phism, but not vice versa. 
\end{proposition}

\begin{proof} This follows from exactly by the same arguments given in Propositions 2.4.2 and 2.4.3 of \cite{LV12}, where this is shown when $R$ is a field. We assume flatness since, for general rings $R$, the bar construction of dg $R$-algebras will not preserve quasi-isomorphisms, but it does when restricted to dg $R$-flat algebras, such as the cobar construction of a connected dg $R$-flat coalgebra.
\end{proof}

The two notions  of $\Omega$-quasi-isomorphism and quasi-isomorphism agree on simply connected dg $R$-flat coalgebras, namely, dg $R$-flat coalgebras $C$ such that $C_1=0$ and $C_0 \cong R$.

\begin{proposition}\label{simpconnected} Let $C$ and $C'$ be simply connected dg $R$-flat coalgebras. Then $f: C \to C'$ is a quasi-isomorphism if and only if $f$ is an $\Omega$-quasi-isomorphism.
\end{proposition}
\begin{proof} This follows from a standard spectral sequence argument, see Proposition 2.2.7 of \cite{LV12} (in their terminology $2$-connected means simply connected). 
\end{proof}

We say that two (connected) simplicial cocommutative coalgebras $C$ and $C'$ are ($\Omega$-)\textit{quasi-isomorphic} if there is a zig-zag of ($\Omega$-)quasi-isomorphisms of (connected) simplicial cocommutative coalgebras between $C$ and $C'$. 

\subsection{Brown's twisted tensor product}
We recall Brown's definition of twisting cochains and twisted tensor products. 
Given any $C=(C,\partial_C, \Delta_C) \in \mathbf{dgCoalg}^0_{R}$ and $(A, d_A, \mu_A) \in  \mathbf{dgAlg}_{R}$, the graded $R$-module $\text{Hom}_{R}(C,A)$ becomes a graded associative algebra with \textit{convolution product} 
$$\star: \text{Hom}_{R}(C,A) \otimes \text{Hom}_{R}(C,A) \to \text{Hom}_{R}(C,A)$$ given by the formula $$f \star g = \mu_A \circ (f \otimes g) \circ \Delta_C.$$ A \textit{twisting cochain} is defined to be a linear map $\tau: C \to A$ of degree $-1$ satisfying $$\partial_{Hom} \tau + \tau \star \tau=0.$$ We also require that the compositions $C \xrightarrow{\tau} A \xrightarrow{a} R$ and $R \xrightarrow{c} C \xrightarrow{\tau} A$ are both zero, where $a$ is the augmentation of $A$ and $c$ the coaugmentation of $C$. Given any left dg $A$-module $(M,d_M)$ define
$\partial^{\tau}: C \otimes M \to C \otimes M$ by 
\begin{equation}\label{twisteddiff}
\partial^{\tau} ( x \otimes m)= \partial_C(x) \otimes m + (-1)^{|x|}x \otimes d_M(m) + \sum_{(x)}(-1)^{|x'|} x' \otimes (\tau(x'') \cdot m)
\end{equation}
where we have written $\Delta_C(x)=\sum_{(x)} x' \otimes x''$ using Sweedler notation. It follows that $\partial^{\tau} \circ \partial^{\tau}=0$, so $(C \otimes M, \partial^{\tau})$ is a chain complex called \textit{Brown's twisted tensor product} of $C$ and $M$, which we denote simply by $C \otimes_{\tau} M$. 

The twisted tensor product construction was originally introduced in \cite{B59} to model the singular chain complex of the total space of a fibration in terms of the chains in the base and the chains in the fiber, see the main statement of \cite{B59}.

For any any $C\in \mathbf{dgCoalg}^0_{R}$ the natural map $\iota: C \rightarrowtail \overline{C} \cong s^{-1} \overline{C} \hookrightarrow \Omega C$ is an example of a twisting cochain called the \textit{universal twisting cochain} of $C$. We now prove the invariance of Brown's twisted tensor product with respect to $\Omega$-quasi-isomorphisms of connected dg coalgebras in the following sense.

\begin{theorem} \label{invarianceofbrown} Let $C$ and $C'$ be two connected dg  $R$-flat coalgebras and let $M$ be a left dg $\Omega(C')$-module. Any $\Omega$-quasi-isomorphism $g: C \to C'$ induces a quasi-isomorphism of chain complexes
$$g \otimes id: C \otimes_{\iota} \Omega(g)^*M \to C' \otimes_{\iota} M,$$
where $\Omega(g)^*M$ denotes the dg $R$-module $M$ equipped with the left dg $\Omega(C)$-module structure obtained by pulling back the left $\Omega(C')$-module structure on $M$ via $\Omega(g): \Omega(C) \to \Omega(C')$
\end{theorem}
Theorem \ref{invarianceofbrown} will follow from Propositions \ref{bar complex} and \ref{invarianceofbar} below, which use techniques and constructions similar to those appearing in \cite{HMS74}, \cite{LH03}, and \cite{P11}. 

\begin{proposition}  \label{bar complex} Let $C$ be a connected dg coalgebra and denote by $\iota: C \to \Omega(C)$ the universal twisting cochain. If $M$ is any left dg $\Omega(C)$-module, then there is a natural quasi-isomorphism of chain complexes 
\begin{equation}
\phi:  B(R,\Omega( C ), M) \to C \otimes_{\iota} M
\end{equation}
\end{proposition}
\begin{proof}
Define $\phi:  B(R, \Omega( C ), M) \to C \otimes_{\iota} M$ by setting
$\phi( [a_1 | ... |a_n]  \otimes m) =0$ if $n>1$, $\phi=id$ if $n=0$, and if $n=1$ with $a_1=\{c_1| ... |c_k\}$ let
\[
\phi( [ \{c_1 | ... |c_k\} ] \otimes m)=\left\{
  \begin{array}{lll}
    c_1 \otimes m                                        &  k=1,\\
  c_1\otimes \{c_2|...|c_k\}\cdot m, & k>1;
  \end{array}
\right.
\]
It is straightforward to check $\phi$ is a chain map. Moreover, $\phi$ is surjective with right inverse given by the chain map $$ \rho_C \otimes id_M:C \otimes_{\iota} M \to  B(R,\Omega( C ), M),$$ where $\rho_C: C \to B\Omega(C)=B(R, \Omega(C), R)$ is the dg coalgebra map defined by
\begin{equation}
\rho_C(c)= [ \{c \} ]+ \sum_{(c)} [ \{c'\}  | \{c''\} ] + \sum_{(c)} [ \{c' \}| \{c'' \}| \{c'''\} ] + ... \text{  , }
\end{equation}
and the number of prime subscripts denotes the number of iterated applications of $\Delta: C \to C \otimes C$; this notation is unambiguous since $C$ is coassociative. Note that $\rho_C$ is well defined since $C$ is connected and thus conilpotent. 

We argue that $(\text{ker } \phi, d_{\text{bar}})$ is an acyclic sub-complex in order to conclude that $\phi$ is a quasi-isomorphism. In fact, define $h: \text{ker } \phi \to \text{ker } \phi$ on any $ [a_1 | ... |a_n ] \otimes m \in \text{ker } \phi$ with $a_n=\{c_1 | ... |c_k\} \in \Omega(C)$ by 

\[
h( [ a_1|a_2|...|a_{n-1}| \{c_1|...|c_k\}] \otimes m)=\left\{
  \begin{array}{lll}
    0,                                              &  k=1,\\
                              \text{$  [a_1|a_2|...|a_{n-1}| \{c_1\} | \{c_2|...|c_k\}] $}      \otimes m & k>1;
  \end{array}
\right.
\]
A computation yields that, since $C$ is conilpotent, for any $x \in \text{ker } \phi$ there exists a non-negative integer $n_x$ such that $(d_{\text{bar}} \circ h + h \circ d_{\text{bar}} - id)^{n_x}=0$. This last equation implies that if $x \in \text{ker } \phi$ is a cycle then there exists some $y$ such that $x=d_{\text{bar}}(y)$, as desired. 
\end{proof}
We adapt the argument from Proposition 2.2.4 of \cite{LV12} to prove the bar construction is invariant under quasi-isomorphisms in the following sense. 
\begin{proposition}\label{invarianceofbar}
If $f:A\to A'$ is a quasi-isomorphism of dg augmented $R$-flat algebras and $M$ is a dg $A'$-module then
$$B(id_{R},f, id_{M}): B(R,A, f^*M) \to  B(R, A', M)$$
is a quasi-isomorphism of chain complexes. 
\end{proposition}
\begin{proof}
Consider the filtration defined by  $$F_p(B(R, A, f^*M))= \{ [a_1|...|a_n] \otimes m: n\leq p \}$$ and define $F_p(B(R, A', M))$ similarly. These are increasing, bounded below, and exhaustive filtrations of chain complexes so they yield convergent spectral sequences. The desired result follows by noting that $B(id_{R},f, id_{M})$ induces a chain map on the associated quotients
$$F_p(B(R, A, f^*M)) / F_{p-1}(B(R, A, f^*M)) \to F_p(B(R, A', M)) / F_{p-1}(B(R, A', M)),$$
which is a quasi-isomorphism by K\"unneth's theorem, since we assumed that $A$ and $A'$ are $R$-flat.
\end{proof} \text{ }
\\
\textit{Proof of Theorem \ref{invarianceofbrown}.} Let $g: C \to C'$ be an $\Omega$-quasi-isomorphism between connected dg coassociative $R$-flat coalgebras. Then $\Omega(g):\Omega C \to \Omega C'$ is a quasi-isomorphism of dg associative $R$-flat algebras. Consider the following commutative square
\[
\xymatrix{
B(R, \Omega( C ),  \Omega(g)^*M) \ar[r]^-{\phi \otimes id} \ar[d]_{B(id_{R}, \Omega(g), id_{M})}& C \otimes_{\iota} \Omega(g)^*M \ar[d]^{g \otimes id} \\
B(R, \Omega( C' ), M)  \ar[r]^-{\phi \otimes id} & C' \otimes_{\iota}M .
}
\]
The horizontal maps are quasi-isomorphisms by Proposition \ref{bar complex}. The left vertical map is a quasi-isomorphism by Proposition \ref{invarianceofbar}. Hence, it follows that the right vertical map is a quasi-isomorphism as well by the $2$ out of $3$ property.  \hfill $\square$

\section{Simplicial coalgebras and localization}
In this section we recall some results from \cite{G95} relating simplicial coalgebras and Bousfield localization and then discuss the notion of a $\pi_1$-$R$-\textit{equivalence} between reduced Kan complexes.

Let $\textbf{sSet}$ denote the category of simplicial sets. We say $S \in \textbf{sSet}$ is a $0$-\textit{reduced} simplicial set if $S$ has a single vertex, i.e. if the set $S_0$ is a singleton. Denote by $\textbf{sSet}^0 \subset \textbf{sSet}$ the full sub-category consisting of reduced simplicial sets.  

\begin{definition} Let $X, Y \in \textbf{sSet}$. A map $f: X \to Y$ is an $R$-\textit{equivalence} if $H_*(f;R): H_*(X;R) \to H_*(Y;R)$ is an isomorphism. We say that $X$ and $Y$ are $R$-\textit{equivalent} if there is a zig-zag of $R$-equivalences in $\mathbf{sSet}$ connecting $X$ and $Y$. 
\end{definition}

Any (Kan) weak homotopy equivalence of simplicial sets is an $R$-equivalence, but not vice-versa. Bousfield constructed in \cite{B75} a model category structure on $\textbf{sSet}$ whose weak equivalences are the $R$-equivalences and cofibrations are the same as those in Quillen's model structure on $\mathbf{sSet}$ (the level-wise injections). A fibrant replacement $X \to L_{R}X$ in Bousfield's model structure on $\textbf{sSet}$ yields a model for the \textit{$R$-localization} of $X \in \mathbf{sSet}$.

For any $X \in \textbf{sSet}$ denote by $R X \in \textbf{scCoalg}_{R}$ the \textit{simplicial cocommutative $R$-coalgebra of chains in $X$}, namely, each $(R X)_n := R[X_n]$ is the free $R$-module generated by $X_n$, the face and degeneracy maps are induced by those in $X$, and each coproduct $$\Delta_n: (RX)_n \to (RX)_n\otimes (RX)_n$$ is defined on basis elements $x \in X_n$ by $$\Delta_n(x)= x \otimes x.$$ Note that the coproduct is induced by the diagonal map $X \to X \times X$. The counit is induced by the map $X \to \Delta^0$. This construction defines a functor 
$$R: \textbf{sSet} \to \textbf{scCoalg}_{R}.$$ The functor $R$ has a right adjoint $$\mathcal{P}: \textbf{scCoalg}_{R} \to \textbf{sSet},$$ called the \textit{functor of points}, whose $n$-simplices are given by $$(\mathcal{P}(C))_n= \text{Hom}_{\textbf{Coalg}_{R}}(R, C_n).$$

When $R$ is a field, which we denote by $\mathbb{F}$, Goerss constructed in \cite{G95} a cofibrantly generated model category structure on $\textbf{scCoalg}_{\mathbb{F}}$ with weak equivalences given by quasi-isomorphisms of simplicial cocommutative $\mathbb{F}$-coalgebras (as defined in Section 2.3) and cofibrations given by level-wise inclusions. Raptis extended this model category structure to simplicial cocommutative $R$-coalgebras over an arbitrary unital commutative ring $R$ \cite{R13}. 

The adjunction $(\mathbb{F}, \mathcal{P})$ becomes a Quillen adjunction when $\textbf{sSet}$ is equipped with the model category structure constructed by Bousfield in \cite{B75}. For any $X \in \mathbf{sSet}$, the derived unit $\eta: X \to \mathbf{R}\mathcal{P}(\mathbb{F}X)$, where $\mathbf{R}\mathcal{P}$ denotes the total derived functor of $\mathcal{P}$, gives a canonical map from $X$ to a fibrant object in Bousfield's model category (i.e. an ``$\mathbb{F}$-local space"). Furthermore, using that the category of sets is an idempotent retract of the category of coalgebras over a fixed algebraically closed field, Goerss proves the following theorem.

\begin{theorem} (Theorem C in \cite{G95}) 
If $\mathbb{E}$ is an algebraically closed field, then for any $X \in \textbf{sSet}$, the derived unit  $$\eta: X \to \mathbf{R}\mathcal{P}(\mathbb{E}X)$$ is the Bousfield localization of $X$. 
\end{theorem}
Using that any field extension $\mathbb{F} \subseteq \mathbb{E}$ induces a weak homotopy equivalence $$L_{\mathbb{F}}X \xrightarrow{\simeq} L_{\mathbb{E}}X$$ between Bousfield localizations, Goerss obtained, as a consequence of the previous theorem, that the simplicial cocommutative coalgebra of chains over \textit{any} field classifies spaces up to Bousfield localization in the following sense.

\begin{theorem} (Theorem D in \cite{G95}) Let $\mathbb{F}$ be any field and $X, Y \in \mathbf{sSet}$. The simplicial cocommutative coalgebras of chains $\mathbb{F}X$ and $\mathbb{F}Y$ are quasi-isomorphic if and only if $X$ and $Y$ are $\mathbb{F}$-equivalent. 
\end{theorem}

One of the main goals of this article is to relate the notion of $\Omega$-quasi-isomorphism between simplicial cocommutative coalgebras to the following notion.

\begin{definition} Let $X, Y \in \textbf{sSet}^0$ be two reduced Kan complexes. A map $f: X \to Y$ is a \textit{$\pi_1$-$R$-equivalence} if it induces an isomorphism $$\pi_1(f): \pi_1(X) \xrightarrow{\cong} \pi_1(Y)$$ between fundamental groups and the induced map at the level of universal covers $$\widetilde{f}: \widetilde{X} \to \widetilde{Y}$$ is an $R$-equivalence. We say $X$ and $Y$ are \textit{$\pi_1$-$R$-equivalent} if there is a zig-zag of $\pi_1$-$R$-equivalences of reduced Kan complexes connecting $X$ and $Y$. 
\end{definition}

The following is analogous to Proposition \ref{cobarstrong}. 

\begin{proposition} Any $\pi_1$-$R$-equivalence between reduced Kan complexes is an $R$-equivalence but not vice-versa. 
\end{proposition}

\begin{proof} Suppose $f: X \to Y$ is a $\pi_1$-$R$-equivalence between Kan complexes, so that $\pi_1(f): \pi_1(X) \cong \pi_1(Y):= \pi_1$ is an isomorphism and $$C_*(\widetilde{f}; R): C_*(\widetilde{X}; R) \to C_*(\widetilde{Y};R)$$ is a quasi-isomorphism of chain complexes, where, for any simplicial set $S$, we denote by $C_*(S; R)= N_*(R S)$, the normalized simplicial chains on $S$ with coefficients in $R$. Let $R[\pi_1]$ be the group algebra on $\pi_1$ and consider $R$ as a left $R[\pi_1]$-module through the natural augmentation $R[\pi_1] \to R$. We have a natural isomorphism of chain complexes $$C_*(X;R) \cong C_*(\widetilde{X}; R) \otimes_{R[\pi_1]} R$$ and similarly for $Y$. But $C_*(\widetilde{X};R)$ is a free $R[\pi_1]$-module, which implies $C_*(\widetilde{X}; R) \otimes_{R[\pi_1]} R$ is a model for the derived tensor product of $R[\pi_1]$-modules, so $C_*(f;R) = C_*(\widetilde{f};R) \otimes_{R[\pi_1]} id_{R}$ is a quasi-isomorphism. Clearly, the converse is not true since an $R$-equivalence does not necessarily induce an isomorphism on fundamental groups. 
\end{proof}

 Let $\mathbb{E}$ be an algebraically closed field and let $\mathcal{R}: \textbf{scCoalg}_{\mathbb{E}} \to \textbf{scCoalg}_{\mathbb{E}}$ be a fibrant replacement functor in Goerss' model category structure on  $\textbf{scCoalg}_{\mathbb{E}}$ so that $$X \to (\mathcal{P} \circ \mathcal{R}) (\mathbb{E}X)$$ is a functorial model for the Bousfield $\mathbb{E}$-localization of $X$. 

 For any Kan complex $X \in  \textbf{sSet}^0$  with universal cover $\widetilde{X}$, we have a natural fibration 
\begin{equation}\label{borelfibration}
(\mathcal{P} \circ \mathcal{R})( \mathbb{E}\widetilde{X}) \to E\pi_1(X) \times_{\pi_1(X)} (\mathcal{P} \circ \mathcal{R})( \mathbb{E}\widetilde{X})  \to B\pi_1(X),
\end{equation}
where $E\pi_1(X) \to B\pi_1(X)$ is a functorial model for the universal bundle of the group $\pi_1(X)$. In other words,  \ref{borelfibration} is the Borel fibration associated to the $\pi_1(X)$ action on $(\mathcal{P} \circ \mathcal{R})( \mathbb{E}\widetilde{X})$. The fibration \ref{borelfibration} is the \textit{fiberwise $\mathbb{E}$-localization} of the fibration 
\begin{equation}
\widetilde{X} \to  E\pi_1(X) \times_{\pi_1(X)} \widetilde{X} \to B\pi_1(X),
\end{equation} whose homotopy class classifies the $\pi_1(X)$-space $\widetilde{X}$. For simplicity, we denote \ref{borelfibration} by
\begin{equation}
L_{\mathbb{E}}\widetilde{X} \to E_{\mathbb{E}}(\widetilde{X}) \to B\pi_1(X)
\end{equation}

The following proposition is now straightforward.
\begin{proposition} \label{fiberwiseequivalence} Let $\mathbb{F}$ be a field with algebraic closure $\mathbb{E}$. A map $f: X \to Y$ between reduced Kan complexes is a  \textit{$\pi_1$-$\mathbb{F}$-equivalence} if and only if $\pi_1(f): \pi_1(X) \xrightarrow{\cong} \pi_1(Y)$ is an isomorphism and $f$ induces a commutative diagram

\[
\xymatrix{
L_{\mathbb{E}}\widetilde{X} \ar[r] \ar[d]_{\simeq} & E_{\mathbb{E}}(\widetilde{X}) \ar[d]^{\simeq} \ar[r] &  B\pi_1(X)\ar[d]^{\cong } \\
 L_{\mathbb{E}}\widetilde{Y}\ar[r]  & E_{\mathbb{E}}(\widetilde{Y}) \ar[r]  & B\pi_1(Y),
}
\]
where the first two vertical arrows are weak homotopy equivalences. 
\end{proposition}

\begin{remark} The above notion of $\pi_1$-$R$-equivalence between spaces can also be described by applying the \textit{fiberwise $R$-completion} construction, as introduced in \cite{BK71} and \cite{BK72}, to the fibration $\widetilde{X} \to X \to B\pi_1(X)$ for any reduced Kan complex $X$. This follows since a map of spaces is an $R$-equivalence if and only if it is a weak homotopy equivalence between $R$-localizations if and only if it is a weak homotopy equivalence between $R$-completions \cite{BK72}. In fact, something stronger is true for simply connected spaces $Z$: the $R$-completion $Z \to R_{\infty}Z$ is equivalent to the $R$-localization of $Z$, for a subring $R$ of $\mathbb{Q}$, or the field of $p$ elements $R=\mathbb{Z}_p$, as discussed in \cite{B75}. Hence, for fibrations with simply connected fiber, e.g. $\widetilde{X} \to X \to B\pi_1(X)$, the fiberwise $R$-completions and fiberwise $R$-localizations agree.

In the proof of our main theorem in section 6, we use the fact that the Bousfield $R$-localization of a space can be assumed to be given by a functorial construction at the level of simplicial sets before passing to the homotopy category. This follows since fibrant replacements may be taken to be functorial in Bousfield's model category structure as a consequence of the small object argument used in \cite{B75}. The completion and its fiberwise version are also functorial constructions as described in \cite{BK71}. 
\end{remark}

\section{The simplicial twisted tensor product}

In this section we introduce the notion of \textit{simplicial twisted tensor product} between a simplicial cocommutative coalgebra and a simplicial associative algebra. More precisely, given a simplicial cocommutative coalgebra $C$, a simplicial associative algebra $A$, and a \textit{simplicial twisting cochain} $\tau:C \rightarrow A$, we construct a simplicial $R$-module $C \otimes_\tau A$, which is compatible with the classical twisted Cartesian product construction. If we further assume that $A$ has the structure of a simplicial cocommutative bialgebra compatible with the simplicial twisting cochain, then $C \otimes_\tau A$ inherits a simplicial cocommutative coalgebra structure. This construction will be used to define the universal cover of a simplicial cocommutative coalgebra in section 5.

\subsection{Notation}

Throughout sections 4, 5, and 6 we assume that $R$ is an integral domain, whenever we say ``coalgebra" and ``algebra" we mean ``coassociative counital $R$-coalgebra"  and ``associative unital $R$-algebra", respectively.

In order to distinguish certain factors in the coproduct of a simplicial cocommutative coalgebra, we introduce a Sweedler style notation to distinguish certain factors in the coproduct:
if $C$ is a simplicial cocommutative coalgebra we denote each coproduct $\Delta_n:C_n \rightarrow C_n \otimes C_n$ by $$\Delta_n(x)=\sum_{(x)} \widetilde{x}\otimes \bar{x}$$ for any $x \in C_n$. We will sometimes omit the subscript when the context is clear and write $\Delta_n(x)= \sum \widetilde{x}\otimes \bar{x}$ and if there are several coproducts involved in a calculation we write $\Delta_n(x)=\sum_{\Delta} \widetilde{x}\otimes \bar{x}$. 

Using this notation the identity $\Delta_{n-1}(d_i(x))=(d_i \otimes d_i)\Delta_n(x)$ may be written as
\[
\sum d_i \widetilde{x} \otimes d_i \bar{x}= \sum \widetilde{d_i x } \otimes \overline{d_i x}.
\]
The equation for coassociativity may be written as
\[
(id \otimes \Delta)\Delta(x)=\sum \sum \widetilde{x} \otimes \widetilde{\overline{x}} \otimes \overline{\overline{x}} =\sum \sum \widetilde{\widetilde{x}} \otimes \overline{\widetilde{x}} \otimes \overline{x} =(\Delta \otimes id)\Delta(x).
\]

\subsection{Simplicial twisting cochains and simplicial twisted tensor product}

\begin{definition}
For any two simplicial $R$-modules  $G$ and $H$, the \textit{simplicial tensor product} of $G \otimes H$ is defined as 

\[(G \otimes H)_n : = G_n \otimes H_n,
\]
with face and degeneracy maps given by
\[
d_i^{G \otimes H}=d_i^G \otimes d_i^H
\]
and
\[
s_i^{G \otimes H}=s_i^G \otimes s_i^H.
\]
\end{definition}

For notational simplicity we will from now on drop the superscripts $G$, $H$ and $G\otimes H$ from the face and degeneracy maps.

\begin{definition}\label{def:MCequationcoalg} 
Let $(C, \Delta)$ be simplicial  connected cocommutative coalgebra and $(A, \mu)$ a simplicial associative algebra. A \textit{simplicial twisting cochain} is a degree $-1$ map of graded $R$-modules $\tau: C \to A$, i.e. a collection of linear maps $\{\tau_n: C_n \to A_{n-1}\}_{n \geq1}$, satisfying the following identities:
\begin{enumerate}\label{eq:MCsimplicialabeliangroup}
\item $d_{j-1}\tau=\tau d_j$, for $j\geq 2$
\item $\tau d_1=\mu\circ (d_0\tau \otimes \tau d_0)\circ \Delta$
\item $s_{j-1}\tau =\tau s_j$, for $j \geq 1$
\item $(id_C \otimes \mu)\circ(id_C\otimes \tau s_0 \otimes id_A)\circ (\Delta \otimes id_A)=id_{C\otimes A}$
\end{enumerate}
\end{definition}
We now define the \textit{simplicial twisted tensor product}.
\begin{theorem-definition}\label{prop:twistedabeliangroup} 
Let $(C, \Delta)$ be simplicial cocommutative coalgebra, $(A, \mu)$ a simplicial associative algebra, and $\tau: C \to A$ a simplicial twisting cochain. For any $x\otimes g \in C_n \otimes A_n$ define
\begin{itemize}
\item $d_0^\tau(x \otimes g):= \sum_{(x)} d_0 (\widetilde{x}) \otimes d_0(g)\cdot \tau(\overline{x})$,
\item $d_i^\tau (x\otimes g):= d_i(x) \otimes d_i(g)$ for $i\geq1$,
\item $s_j^\tau (x\otimes g):= s_j(x) \otimes s_j(g)$ for $j\geq0$.
\end{itemize}
The maps $d_i^{\tau}$ and $s_j^{\tau}$ satisfy the simplicial identities and, consequently, define a simplicial $R$-module denoted by $C \otimes_{\tau} A$ with $(C \otimes_{\tau} A)_n = C_n \otimes A_n$. We call $C \otimes_{\tau} A$ the \textit{simplicial twisted tensor product} of $C$ and $A$ with respect to $\tau:C \to A$, and $d_i^{\tau}$ and $s_j^{\tau}$ the twisted face and degeneracy maps.
\end{theorem-definition}

\begin{proof}
To show that $d_i^{\tau}$ and $s_j^{\tau}$ define a simplicial $R$-module structure on $C \otimes_{\tau} A$ we need to check the simplicial identities. Since the simplicial twisting cochain only affects $d_0$, we only need to check the following identities:
\begin{enumerate}
\item $d_0^\tau d_j=d_{j-1}d_0^\tau$,
\item $d_0^\tau s_j=s_{j-1}d_0^\tau$,
\item $d_0^\tau s_0=id$.
\end{enumerate}
The first identity splits into two cases: $j=1$ and $j\geq 2$. We first show the  $j=1$ case, i.e. $d_0^\tau d_1= d_0^\tau d_0^\tau$. Let $x \otimes g \in C_n \otimes A_n$ for some $n$, then we get the following sequence of identities for $d^{\tau}_0d_1$:
\begin{align}
d_0^\tau d_1 (x \otimes g) &= d_0^{\tau}(d_1(x) \otimes d_1(g)) \\
& = \sum d_0(\widetilde{d_1 (x)})\otimes d_0(d_1(g)) \cdot \tau(\overline{d_1(x)}) \\
& = \sum d_0(d_1( \widetilde{x})) \otimes d_0(d_1 (g)) \cdot \tau(d_1( \overline{x})) \\
& = \sum d_0(d_0( \widetilde{x}) )\otimes d_0(d_0( g)) \cdot \tau(d_1(\overline{x})). \label{eq:MC4.4}
\end{align}
In the second line above we have used that $d_1$ is a coalgebra map and in the third line that $d_0d_1=d_0d_0$. On the other hand, we also have:
\begin{align}
d_0^\tau d_0^\tau (x \otimes g) & = \sum d^\tau_0 (d_0(\widetilde{x})\otimes d_0(g) \cdot \tau(\overline{x})) \\
&= \sum \sum d_0 (\widetilde{d_0(\widetilde{x})}) \otimes  d_0(d_0 (g) \cdot \tau(\overline{x})) \cdot \tau(\overline{d_0(\widetilde{x})}) \\
&= \sum \sum d_0 (d_0(\widetilde{\widetilde{x}})) \otimes  d_0(d_0 (g)) \cdot d_0(\tau(\overline{x})) \cdot \tau(d_0(\overline{\widetilde{x}})) \\
&= \sum \sum d_0 (d_0(\widetilde{x})) \otimes  d_0(d_0 (g)) \cdot d_0(\tau(\widetilde{\overline{x}})) \cdot \tau(d_0(\overline{\overline{x}})). \label{eq2} \\
&= \sum d_0 (d_0(\widetilde{x}) )\otimes  d_0(d_0 (g)) \cdot \left(\sum d_0\tau((\widetilde{\overline{x}})) \cdot  \tau(d_0(\overline{\overline{x}}))\right).\label{eq:MC4.9} 
\end{align}
In the second line we used the fact that $d_0$ is a coalgebra map and an algebra map and in the third line we used the coassociativity of the coproduct. The equality of \ref{eq:MC4.4} and \ref{eq:MC4.9} follows from the definition of a simplicial twisting cochain.

We now verify that $d_0^\tau d_{j}= d_{j-1} d_0^\tau$ for $j\geq 2$ by using the fact that $d_0$ and $d_{j-1}$ are algebra and coalgebra morphisms and the identities $d_{j-1}d_0=d_0d_j$, $d_{j-1}\tau=\tau d_j$. 
\begin{align}
d_{j-1} d_0^\tau (x\otimes g) &=d_{j-1}(\sum d_0(\widetilde{x}) \otimes d_0(g) \cdot \tau(\overline{x})) \\ 
& = \sum d_{j-1}(d_0(\widetilde{x}) )\otimes d_{j-1}(d_0(g)) \cdot d_{j-1}(\tau(\overline{x})) \\
&= \sum d_0 (d_j(\widetilde{x})) \otimes d_0(d_j(g)) \cdot \tau(d_j (\overline{x})) \\
& = \sum d_0(\widetilde{d_j(x)} )\otimes d_0(d_j(g)) \cdot \tau(\overline{ d_j(x)}) \\
&= d_0^{\tau}(d_j(x\otimes g)).
\end{align}

We continue by checking the identity $d_0^\tau s_j=s_{j-1}d_0^\tau$ as follows:
\begin{align}
d_0^\tau s_j (x \otimes g) & = d^\tau_0(s_j(x) \otimes s_j(g)) \\
& = \sum d_0(\widetilde{s_j(x)}) \otimes d_0(s_j(g)) \cdot \tau (\overline{s_j(x)}) \\
&= \sum d_0(s_j\widetilde{(x)}) \otimes d_0(s_j(g)) \cdot \tau (s_j(\overline{x})) \\
&=\sum s_{j-1}(d_0(\widetilde{x})) \otimes s_{j-1}(d_0(g)) \cdot s_{j-1} \tau(\overline{x}) \\
&= s_{j-1}d_0^\tau(x \otimes g).
\end{align}
In the second line, we wrote down the definition of the twisted face map. In the third line we used $s_j$ is a coalgebra map and in the fourth line we used $d_0s_j=s_{j-1}d_0$ and $\tau s_j=s_{j-1} \tau$.

To show the last identity $d_0^\tau s_0=id$, we note
\begin{align}
d_0^\tau s_0 (x \otimes g)&= d^\tau_0 (s_0 (x) \otimes s_0(g)) \\
&= \sum d_0 (\widetilde{s_0 (x)}) \otimes d_0 (s_0(g) )\cdot \tau ( \overline{s_0 (x)}) \\
&= \sum d_0 (s_0 (\widetilde{ x})) \otimes d_0( s_0(g)) \cdot \tau (s_0 (\overline{ x})) \\
&= \widetilde{x} \otimes g \cdot \tau(s_0(\overline{x})) \\
&= x \otimes g.
\end{align}
In the second line, we use the definition of the twisted face map and in the third line, we used the fact that $s_0$ is a coalgebra map. In the fourth line, we used $d_0s_0=id$ and in the fifth line, we used equation 4 of Definition \ref{def:MCequationcoalg}.

From these calculations it follows that the twisted face and degeneracy maps define a simplicial $R$-module $C \otimes_\tau A$.
\end{proof}

\subsection{The simplicial coalgebra structure on the simplicial twisted tensor product}\label{sec:twistedcocomcoalg}

Suppose $A$ is a simplicial algebra equipped with a simplicial cocommutative coalgebra structure making it into a simplicial bialgebra. We will show that if $\tau: C \to A$ is a simplicial twisting cochain which is compatible with the cocommutative coalgebra structure, then the coproducts of $C$ and $A$ induce a simplicial cocommutative coalgebra structure on the simplicial twisted tensor product $C \otimes_{\tau} A$. 

\begin{definition}
Let $C$ be a connected simplicial cocommutative coalgebra and $A$ a simplicial cocommutative bialgebra. A simplicial twisting cochain $\tau: C \to A$ is called a \textit{simplicial coalgebra twisting cochain} if $\tau$ is a coalgebra map, i.e.
$$
\Delta_{n-1} ' \tau = (\tau \otimes \tau)\Delta_n,
$$
where $\Delta_n:C_n \rightarrow C_n \otimes C_n$ is the coproduct of $C$ and $\Delta_{n-1}':A_{n-1} \rightarrow A_{n-1} \otimes A_{n-1}$ is the coproduct of $A$. 
\end{definition}

\begin{proposition}\label{coalgebraontwistedproduct} Let $C$ be a connected simplicial cocommutative coalgebra, $A$  a simplicial cocommutative bialgebra, and $\tau: C\to A$ a simplicial coalgebra twisting cochain. Then the  twisted tensor product  $C\otimes_\tau A$ becomes a simplicial cocommutative coalgebra with coproduct given by 
\[
\Delta_{C\otimes A}:C\otimes A \rightarrow C\otimes A \bigotimes C \otimes A
\]
\[
\Delta_{C\otimes A}:=(id \otimes T \otimes id)(\Delta_C \otimes \Delta_A),
\]
where $T:C \otimes A \rightarrow A \otimes C$ is the flip map. 
\end{proposition}

\begin{proof}
From Proposition \ref{prop:twistedabeliangroup} it follows that $C \otimes _\tau A$ is a simplicial $R$-module, so we must show that $\Delta_{C\otimes A}$ defines a simplicial cocommutative coalgebra structure, i.e. 
\[
\Delta_{C\otimes A} (d_0^{\tau})=(d_0^{\tau} \otimes d_0^{\tau})\Delta_{C \otimes A}.
\]
Since the degeneracy and face maps $d_i$ for $i \geq 1$ commute with the coproduct, these maps satisfy the simpicial identities. We check the compatibility of the coproduct with $d^\tau_0$. For any $x \otimes g \in C \otimes_\tau A$ we have
\begin{align}
    \Delta_{C\otimes A} (d_0^{\tau})(x\otimes g) &= \sum_{\Delta_C} \Delta_{C \otimes A}\left(d_0(\widetilde{x})\otimes d_0(g) \cdot \tau(\overline{x})\right)   \\
    &=(id \otimes T \otimes id)\left(\sum_{\Delta_A} \Delta_C(d_0(\widetilde{x}))\otimes \Delta_{A}( d_0(g) \cdot \tau(\overline{x}))\right) \\
    &=\sum_{\Delta_C}\sum_{\Delta_C} \sum_{\Delta_A} \widetilde{d_0(\widetilde{x})}\otimes \widetilde{d_0(g) \cdot \tau(\overline{x})} \bigotimes \overline{d_0(\widetilde{x})}\otimes \overline{d_0(g) \cdot \tau(\overline{x})}  \\
    & = \sum_{\Delta_C} \sum_{\Delta_C} \sum_{\Delta_C} \sum_{\Delta_A}     \widetilde{d_0 (\widetilde{x})}\otimes \widetilde{d_0(g)} \cdot \widetilde{\tau (\overline{x})} \bigotimes \overline{d_0(\widetilde{x})} \otimes \overline{d_0(g)} \cdot \overline{\tau (\overline{x})} \\
    & = \sum_{\Delta_C} \sum_{\Delta_C} \sum_{\Delta_C} \sum_{\Delta_A}     \widetilde{d_0 (\widetilde{x})}\otimes \widetilde{d_0(g)} \cdot \tau(\widetilde{\overline{x}}) \bigotimes \overline{d_0(\widetilde{x})} \otimes \overline{d_0(g)} \cdot \tau(\overline{\overline{x}}) \\
    & = \sum_{\Delta_C} \sum_{\Delta_C} \sum_{\Delta_C} \sum_{\Delta_A} d_0(\widetilde{\widetilde{x}})\otimes \widetilde{d_0(g)} \cdot \tau(\widetilde{\overline{x}}) \bigotimes d_0(\overline{\widetilde{x}}) \otimes \overline{d_0(g)} \tau(\overline{\overline{x}}) \\
    & = \sum_{\Delta_C} \sum_{\Delta_C} \sum_{\Delta_C} \sum_{\Delta_A} d_0(\widetilde{\widetilde{\widetilde{x}}}) \otimes \widetilde{d_0(g)} \cdot \tau(\overline{\widetilde{x}}) \bigotimes d_0(\overline{\widetilde{\widetilde{x}}}) \otimes  \overline{d_0(g)} \cdot \tau(\overline{x}) \\
    & = \sum_{\Delta_C} \sum_{\Delta_C} \sum_{\Delta_C} \sum_{\Delta_A} d_0(\widetilde{\widetilde{x}}) \otimes \widetilde{d_0(g)} \cdot \tau(\widetilde{\overline{x}})\bigotimes d_0(\overline{\widetilde{x}}) \otimes \overline{d_0(g)} \cdot \tau(\overline{\overline{x}})  
\end{align}
\begin{align}
    & = \sum_{\Delta_C} \sum_{\Delta_C} \sum_{\Delta_C} \sum_{\Delta_A} d_0(\widetilde{\widetilde{x}})\otimes \widetilde{d_0(g)} \cdot \tau(\overline{\widetilde{x}} ) \bigotimes d_0(\widetilde{\overline{x}}) \otimes \overline{d_0(g)} \cdot \tau(\overline{\overline{x}})  \\
    & = \sum_{\Delta_C} \sum_{\Delta_C} \sum_{\Delta_C} \sum_{\Delta_A} d_0(\widetilde{\widetilde{x}}) \otimes d_0(\widetilde{g}) \cdot \tau(\overline{\widetilde{x}}) \bigotimes d_0(\widetilde{\overline{x}}) \otimes d_0(\overline{g}) \cdot \tau(\overline{\overline{x}}) \\
    & = \sum_{\Delta_{C \otimes A}} d_0^\tau(\widetilde{x} \otimes \widetilde{g}) \otimes d_0^\tau(\overline{x} \otimes \overline{g}) \\
    & = (d^\tau_0 \otimes d^\tau_0)\Delta_{C\otimes A}(x \otimes g).
\end{align}
We now explain the above calculation. In the third line, we used $\tau$ is a coalgebra map. In the fourth line, we used the bialgebra compatibility. In the fifth line, we used  $\tau: C \to A$ is a coalgebra map and the compatibility of $d^C_0$ with $\Delta_C$. In the sixth line, we used the coproduct $\Delta_C$ is coassociative. In the seventh line, we used that $d_0^C$ is a coalgebra map and, once more, the coassociativity of $\Delta_C$. In the eighth line, we used the coassocaitivity of $\Delta_C$ again. In the ninth line, we used the cocommutativity of $\Delta_C$. In the tenth line, we used the compatibility of $d^A_0$ with $\Delta_A$. In the eleventh line, we used the definition of $d^\tau_0$.

The tensor product of the counits of $C$ and $A$ defines a counit on $C \otimes_{\tau} A$.
\end{proof}

\subsection{The twisted Cartesian product and its relation to the simplicial twisted tensor product}\label{sec:simplicialtwistingmorphisms}

We recall the notion of twisted Cartesian product of simplicial sets and discuss the universal cover of a reduced Kan complex as an example following \cite{C71}. Then we explain the relationship between the twisted Cartesian product and the simplicial twisted tensor product constructions.

\begin{definition}\label{def:MCsimplicialsets}
Let $X$ be a reduced simplicial set and $G$ be a simplicial group. A \textit{twisting morphism} is a degree $-1$ map $t: X \to G$ of graded sets, i.e. a sequence of maps 
$\{t_n:X_n \rightarrow G_{n-1}\}_{n\geq1}$,
satisfying the following identities for any $x \in X_n$
\begin{enumerate}\label{eq:MCsimplicialsets}
 \item     $d_i(t(x))  =t(d_{i+1}(x))$  for $i\geq 1$,
\item   $ d_0(t(x))  = t(d_1(x)) \cdot t(d_0(x))^{-1}$, 
\item    $s_i(t(x))  = t(s_{i+1}(x))$  for $i\geq 1$,
\item   $  e = t(s_0(x))$,
\end{enumerate}
where $e$ denotes the identity element of $G_{n}$. The \textit{twisted Cartesian product of $X$ and $G$ with respect to $t: X \to G$} is the simplicial set defined by
\[(X \times_t G)_n:= X_n \times G_n,
\]
together with face and degeneracy maps given by
 \begin{align}
d_0(x,g) &:=(d_0(x) ,d_0(g) \cdot t(x)), &  \\
s_i(x,g) & := (s_i(x),s_i(g)) ,&\\
d_i(x,g)&:=(d_i(x),d_i(g)) & \mbox{ for $i\geq 1$},
 \end{align}
 for $(x,g)\in (X \times_t G)_n$.
\end{definition}

The twisted Cartesian product construction yields a simplicial model for the universal cover of a reduced Kan complex $X$ as follows. Denote by $G$ the fundamental group of $X$ considered as a discrete simplicial set, i.e. $G_n:=\pi_1(X)$ for all $n \geq 0$ and with the identity as face and degeneracy maps. Define a twisting morphism 
\[
t:X_n \rightarrow G_{n-1},
\]
by 
\[
t(x):=[d_2 ... d_n(x)],
\]
where $[d_2 ... d_n(x)]$ denotes the homotopy class of the $1$-simplex $d_2 ... d_n(x)$ considered as an element of the fundamental group.

\begin{proposition}[\cite{C71}, Example 6.9]\label{prop:simplicialfundamentalgrouptwisting mmorphism}
Let $X$ be a reduced Kan complex and $G$ the fundamental group of $X$ considered as a discrete simplicial group. Then the twisted Cartesian product $\widetilde{X}:=X \times_t G$ is equal to the universal cover of $X$. The right action of $G$ on $\widetilde{X}$ is given by multiplication on the second component of the twisted Cartesian product.
\end{proposition}

We proceed to show that if $X$ is a reduced Kan complex and $G$ is the fundamental group of $X$ seen as a discrete simplicial group, then the twisting morphism $t:X \rightarrow G$  from Section \ref{sec:simplicialtwistingmorphisms}, induces a simplicial coalgebra twisting cochain $\tau= R t:R X \rightarrow R G$ and that we have a natural isomorphism of simplicial cocommutative coalgebras $R (X \times_t G) \cong R X \otimes_\tau R G.$

\begin{lemma}
If $t:X \rightarrow G$ is a twisting morphism from a reduced simplicial set $X$ to a simplicial group $G$, then $\tau=R t: R X \rightarrow R G$ is a simplicial coalgebra twisting cochain.
\end{lemma}

\begin{proof}
We show that $\tau$ satisfies the equations of Definition \ref{def:MCequationcoalg} and that $\tau$ is a coalgebra map.

Equations 1 and 3 from Definition \ref{def:MCequationcoalg} follow immediately from equations 1 and 3 of Definition \ref{def:MCsimplicialsets}, respectively. We show equation 2 of Definition \ref{def:MCsimplicialsets} implies equation 2 of Definition \ref{def:MCequationcoalg}, i.e. $\tau$ satisfies the identity
\begin{equation}\label{eq:MCequationinproof}
\tau d_1=\mu\circ (d_0\tau \otimes \tau d_0)\circ \Delta
\end{equation}
Since $R X$ and $R G$ are free as $R$-modules generated by the sets $X$ and $G$ they have a basis. It is therefore enough to check that  for any basis element  $x \in X_n \subset (R X)_{n}$, we have 
\begin{align}
    &\mu\circ (d_0\tau \otimes \tau d_0)\circ \Delta(x) &=\\ 
    &d_0\tau(g) \cdot \tau d_0 (x)& = \\
    &Rt(d_1(x)) \cdot Rt (d_0(x))^{-1} \cdot Rt (d_0(x))& = \\
    &Rt(d_1(x)) &= \\
    &\tau (d_1(x)),
    \end{align}
where we used that $x$ is group-like in the second line and equation 2 of Definition 9 in the third line.

We now show that equation  4 of Definition \ref{def:MCequationcoalg},
\begin{equation}\label{eq:MC2}
(id_{R X} \otimes \mu)\circ(id_{R X}\otimes \tau s_0 \otimes id_{R G})\circ (\Delta \otimes id_{R G})=id_{RX\otimes R G},
\end{equation}
is also satisfied.  This follows since
\begin{align}
    {}& (id_{R X} \otimes \mu_{R G})\circ(id_{R X}\otimes \tau s_0 \otimes id_{R G})\circ (\Delta_{R X} \otimes id_{R G})(x\otimes g)&= \\
    &(id_{R X} \otimes \mu_{R G})\circ(id_{R X}\otimes \tau s_0 \otimes id_{R G})(x \otimes x \otimes g)& = \\
    &x \otimes \tau s_0 x \cdot g &=\\
    &x \otimes e\cdot g&= \\
  &  x \otimes g
\end{align}
In the first line, we used $x$ is group-like and in the third line we used equation 4 of Definition \ref{def:MCsimplicialsets}.

The fact that $\tau$ is a coalgebra map follows from the fact that both $R X$ and $R G$ have a basis of group-like elements and that $\tau$ preserves these basis elements. Namely, for $x \in X$ a basis element of $R X$, we have 
\begin{align}
  {} & \Delta_{R G}(\tau(x)) & = \\
    &\tau(x) \otimes \tau(x) &= \\
    &(\tau \otimes \tau) \Delta_{R X}(x).
\end{align} \end{proof}
The twisted Cartesian product is compatible with the simplicial twisted tensor product in the following sense.
\begin{proposition} Let $X\times_t G$ be the twisted Cartesian product of a simplicial set $X$ and a simplicial group $G$ with respect to a twisting morphism $t: X_n \to G_{n-1}$. There is an isomorphism of simplicial cocommutative coalgebras \[R(X\times_t G)\cong R X \otimes_\tau R G,\] where $\tau=Rt$.
\end{proposition}

\begin{proof}
Each $R$-module $R_n(X\times_t G)$ has a basis of the form $(x,g)$ with $x \in X_n$ and $g \in G_n$. Similarly $(R X \otimes_\tau R G)_n$ has a basis of the form $x \otimes g$, again with $x \in X_n$ and $g\in G_n$. Using these bases we define a map 
\[
\varphi:R(X\times_t G)\rightarrow R X \otimes_\tau R G,
\]
by setting
\[
\varphi(x,g):=x \otimes g.
\]
We claim the map $\varphi$ induces an isomorphism of graded $R$-modules, commutes with face and degeneracy maps, and commutes with the coalgebra structures.

The fact that $\varphi$ induces an isomorphism of graded $R$-modules is clear because both $R(X\times_t G)$ and $R X \otimes_\tau R G$ are free as $R$-modules and $\varphi$ induces a bijection on the bases.

Note that it is straightforward to show that $\varphi$ commutes with the degeneracy maps and the face maps $d_i$ for $i \geq 1$, so we show that $\varphi(d_0(x,g))=d_0(\varphi(x,g))$ by noting that
\begin{align}
    \varphi(d_0(x,g)) & = \varphi(d_0(x),d_0(g) \cdot t(x)) \\
    & =d_0(x) \otimes d_0(g) \cdot t(x) \\
    &= d_0(x) \otimes d_0(g) \cdot \tau(x) \\
    & = d_0(\varphi(x \otimes g)),
\end{align}
where we used in the second line that on basis elements $\tau(x)$ is defined as $t(x)$ and in the third line that $x$ is group-like. 

Finally,  the fact that $\varphi$ is a coalgebra map follows since all basis elements are group-like, namely
\begin{align}
\Delta(\varphi(x,g)) & = \Delta (x\otimes g) \\
& = x \otimes g \bigotimes x \otimes g \\
& =\varphi(x,g) \bigotimes \varphi(x,g) \\
&= (\varphi \otimes \varphi)\Delta (x,g).
\end{align}
\end{proof}

As an immediate consequence we have the following corollary.

\begin{corollary}\label{cor:universalcover}
Let $X$ be a reduced Kan complex and $G$ be the fundamental group of $X$ seen as a discrete simplicial group. Let $t:X \rightarrow G$ be the twisting morphism of simplicial sets from Proposition \ref{prop:simplicialfundamentalgrouptwisting mmorphism} and denote $\tau= R t$. Then $R X \otimes_\tau R G$ is isomorphic to the simplicial cocommutative coalgebra of chains on the universal cover of $X$. \end{corollary}

\section{The universal cover of a simplicial coalgebra}

Let $C$ be a connected simplicial coassociative $R$-coalgebra. The chain complex of normalized chains $N_*(C)$ becomes a differential graded (dg) connected coassociative coalgebra when equipped with the coproduct given by
$$\delta: N_*(C) \xrightarrow{N_*(\Delta)} N_*(C \otimes C) \xrightarrow{AW} N_*(C) \otimes N_*(C),$$
where $\Delta$ denotes the coproduct of $C$ and $AW$ the Alexander-Whitney map, as previously recalled in section 2. For any such $C$, the cobar construction $\Omega N_*(C)=\Omega(N_*(C), \delta)$ yields a dg associative algebra defining a functor
\[ \Omega N_*: \textbf{sCoalg}^0_{R} \to \textbf{dgAlg}_{R},\]
where $\textbf{sCoalg}^0_{R}$ is the category of connected simplicial coalgebras and $\textbf{dgAlg}_{R}$ the category of dg algebras. 

The backbone for this section comes from the following result, which shows that the dg bialgebra of normalized chains on the based loop space is completely determined by the simplicial cocommutative coalgebra of chains on the underlying Kan complex.

\begin{theorem}\label{nscadams} Let $X$ be a reduced Kan complex with $X_0=\{b\}$ and denote by $C_*(X;R)= N_*(RX)$ the dg coassociative coalgebra of normalized chains on $X$ with $R$-coefficients. Then
\\
1) there exists a natural quasi-isomorphism of dg algebras $$\varphi: \Omega C_*(X;R) \simeq C_*(\Omega_b|X|),$$
where $\Omega_b|X|$ denotes the topological monoid of Moore loops based at $b$ on the geometric realization $|X|$, and
\\
2) there exists a natural coproduct $$\nabla: \Omega C_*(X;R) \to \Omega C_*(X;R) \otimes \Omega C_*(X;R) $$
making $\Omega C_*(X;R)$ a dg bialgebra such that $\varphi$ becomes an quasi-isomorphism of dg bialgebras and for any $\alpha \in (\Omega C_*(X;R))_0$ 
$$\nabla_0( \alpha ) = \alpha \otimes \alpha + 1_{R} \otimes \alpha + \alpha \otimes 1_{R}.$$
\end{theorem}
\begin{proof} Part 1 is an extension of a classical theorem of Adams proven in \cite{RZ16} and \cite{R19} by relating the cobar construction to a cubical version of the left adjoint of the homotopy coherent nerve functor. Part 2 is shown in \cite{RZ19}. 
\end{proof}

A direct consequence of the above result is that $H_0( \Omega C_*(X;R))$ is naturally isomorphic to the fundamental group ring. However, we have decided to include a direct, self contained, and elementary proof for this particular result since it showcases an idea that will be generalized in Section 5.1 for the construction of the fundamental bialgebra of an arbitrary simplicial cocommutative coalgebras. 
 
\begin{theorem}\label{adamscobar} Let $X$ be a reduced Kan complex and denote by $C=RX$ the connected simplicial cocommutative coalgebra of chains. Then there is a natural isomorphism of $R$-algebras
\[H_0(\Omega N_*(C)) \cong R[\pi_1(X)] \]
between the $0$-th homology of the cobar construction of the connected dg coassociative coalgebra $(N_*(C), \delta)$ and the fundamental group algebra of $X$. Consequently, the $R$-algebra $H_0(\Omega N_*(C))$ extends to a cocommutative Hopf algebra whose group-like elements form a group isomorphic to $\pi_1(X)$. 
\end{theorem}
\begin{proof} 
We have $$H_0( \Omega N_*(C) ) = \frac{(\Omega N_*(C))_0 }{ D_{\Omega N_*(C)} ((\Omega N_*(C))_1)}.$$
The free associative algebra $(\Omega N_*(C))_0$ has an $R$-linear basis given by monomials $\{ \bar\sigma_1 | ... | \bar\sigma_n \}$ where  each $\sigma_i$ is a non-degenerate $1$-simplex in $X$ and $\bar\sigma_i$ denotes its class in $N_1(C)=N_1(RX)$. The free multiplication of $(\Omega N_*(C))_0$ is then given by concatenation of monomials. The quotient relation on $(\Omega N_*(C))_0$  that yields the $0$-th homology $R$-algebra is generated by declaring $D_{\Omega N_*(C)} \{ \bar x \}$ to be zero for any non-degenerate $x \in X_2$, or equivalently, by the equivalence relation $$\{ \overline{d_2(x)} | \overline{d_0 (x)}\} \sim -\{\overline{d_0(x)}\} +\{\overline{d_1(x)}\} - \{\overline{d_2(x)}\},$$ where $d_i: X_2 \to X_1$ is the $i$-th face map for $i=0,1,2$.

On the other hand, $R[\pi_1(X)]$, is freely generated as an $R$-module by the elements of $\pi_1(X)$. The group $\pi_1(X)$ may be identified with the groupoid with a single object $\tau_1(X)$, where $\tau_1: \textbf{sSet} \to \textbf{Cat},$
 is the left adjoint of the nerve functor. Recall that $\tau_1(X)$ is given by imposing an equivalence relation on the monoid freely generated by $X_1$. The equivalence relation is generated by declaring $$ \sigma_2 \cdot \sigma_0 \sim' \sigma_1$$ for $\sigma_i \in X_1, i=0,1,2$, if and only if there exists some $x \in X_2$ with $d_i(x)=\sigma_i$ for $i =0,1,2$. For any $\sigma \in X_1$ we denote by $[\sigma]$ the $\sim'$-equivalence class of $\sigma$ in $\pi_1(X)$. The unit of the group algebra $R[\pi_1(X)]$ corresponds to $1_{R} [s_0(*)]$, where $*$ denotes the single vertex of $X$ and $s_0: X_0 \to X_1$ the degeneracy map.

Define a map of $R$-algebras $$\phi: (\Omega N_*(C))_0 \to R[\pi_1(X)]$$
by setting $$\phi: 1_{R} \mapsto 1_{R}[s_0(*)]$$ and
$$\phi: \{\bar\sigma\} \mapsto [\sigma] - 1_{R}[s_0(*)]$$ for any non-degenerate $1$-simplex $\sigma \in X_1$ and then extending $\phi$ as an algebra map to monomials of arbitrary length in $(\Omega N_*(C))_0 $.
We check $\phi$ preserves equivalence relations so it induces a well defined map on homology, which we also denote by $$\phi: H_0(\Omega N_*(C)) \to R[\pi_1(X)].$$ This follows from the computation 
\begin{align*}
&\phi( \{\overline{d_2(x)} | \overline{d_0 (x)}\} + \{\overline{d_0(x)}\} - \{\overline{d_1(x)}\} + \{\overline{d_2(x)}\} ) &=\\
&( [d_2(x)]-1_{R}[s_0(*)])  \cdot ([ d_0 (x)]-1_{R}[s_0(*)]) + [d_0(x)] - & \\ & 1_{R}[s_0(*)] - 
 [d_1(x)]+1_{R}[s_0(*)] +[d_2(x)]- 1_{R}[s_0(*)] &=\\
&[d_2(x)] \cdot [d_0(x) ] -[d_1(x)].&
\end{align*}
The map $ \phi: H_0( \Omega N_*(C)) \to R[\pi_1(X)]$ is clearly an isomorphism of algebras. In fact, the inverse $$\phi^{-1}: R[\pi_1(X)] \to H_0 (\Omega N_*(C))$$ is determined by  
$$\phi^{-1}:1_{R} [s_0(*)] \mapsto 1_{R}$$ 
and
$$\phi^{-1}: [\sigma] \mapsto \{\bar\sigma\} + 1_{R}$$
for any non-degenerate $\sigma \in X_1$. 
Recall that the group algebra $R[\pi_1(X)]$ has a cocommutative coproduct $$\nabla: R[\pi_1(X)] \to  R[\pi_1(X)]\otimes R[\pi_1(X)]$$ determined by setting $$\nabla(g) := g \otimes g$$ for any $g \in \pi_1(X) $. The counit  $\epsilon: R[\pi_1(X)] \to R$ is determined by setting $\epsilon(g):=1_{R}$ for any $g \in \pi_1(X)$. Moreover,  the algebra $R[\pi_1(X)]$ together with the  coproduct $\nabla$ is a cocommutative Hopf algebra with antipode induced by the inverse map $g \mapsto g^{-1}$. Thus $H_0(\Omega N_*(C))$ inherits a cocommutative Hopf algebra structure with coproduct $$\nabla_0:  H_0(\Omega N_*(C)) \to H_0(\Omega N_*(C)) \otimes H_0(\Omega N_*(C))$$ given by
$$\nabla_0:=(\phi^{-1} \otimes \phi^{-1}) \circ \nabla \circ \phi.$$ In fact, the coproduct $\nabla_0:  H_0(\Omega N_*(C)) \to H_0(\Omega N_*(C)) \otimes H_0(\Omega N_*(C))$ is induced from a coproduct $$\nabla_0: (\Omega N_*(C))_0 \to ( \Omega N_*(C))_0  \otimes ( \Omega N_*(C))_0 $$ before passing to homology, which is determined by the formula 
$$ \nabla_0(\alpha)= \alpha \otimes \alpha + 1_{R} \otimes \alpha + \alpha \otimes 1_{R}$$
for any generator $\alpha \in ( \Omega N_*(C))_0$, i.e. any $\alpha=\{\bar\sigma_1| ... |\bar\sigma_n\}$, where each $\sigma_i \in X_1$ is a non-degenerate $1$-simplex. 
\end{proof}

\begin{remark} If $X$ is an arbitrary reduced simplicial set (not necessarily a Kan complex) and $C=RX$, we may recover the fundamental group ring of $X$ from $C$ by incorporating the derived localization of \cite{CHL18} as follows. The same formula for $\nabla_0$ given in the previous proof induces a natural bialgebra structure on $H_0(\Omega N_*(C))$ (this construction did not use the Kan property of $X$). In this case, $H_0(\Omega N_*(C))$ may not be isomorphic as an algebra  to the fundamental group algebra $R[\pi_1(|X|)]$ of the geometric realization of $X$. However, we may consider the derived localization of the dg algebra $\Omega N_*(C)$ at the set of cycles $S \subset ( \Omega N_*(C))_0 $ given by the group like elements of $$\nabla_0:( \Omega N_*(C))_0 \to ( \Omega N_*(C))_0  \otimes ( \Omega N_*(C))_0 . $$ This yields a new dg algebra $\Omega N_*(C)[S^{-1}]$ which is quasi-isomorphic to the chains on the based loop space of $|X|$. Hence $H_0(\Omega N_*(C)[S^{-1}])$ is isomorphic to the fundamental group algebra of $|X|.$ When $X$ is a reduced Kan complex, there is no need to preform this derived localization to obtain the fundamental group algebra. The idea of localizing the cobar construction at a basis of degree $0$-elements, to address the non-simply connected case, was originally treated in \cite{HT10} and briefly described in \cite{Ko09}. Furthermore, the relationship between the  localized cobar construction and the chains on Kan's ``loop group" construction is explained in \cite{HT10}. We expect these constructions to be useful when describing small algebraic models for non-simply connected homotopy types. 
\end{remark}
\subsection{The fundamental bialgebra}\label{sec:coalgebrastructureoncobar}
We now describe how to associate a cocommutative bialgebra to any abstract connected simplicial commutative coalgebra $C$ in such a way that we recover the discussion above when $C=RX$ for some reduced Kan complex $X$.

Let $C$ be a connected simplicial cocommutative coalgebra $C$. Recall that we denote the coproduct $\Delta_1: C_1 \to C_1 \otimes C_1$ by 
\begin{equation}\label{sweedler}
x \mapsto \sum_{(x)} \widetilde{x} \otimes \overline{x}.
\end{equation} Since the degenercy map $s_0: C_0 \to C_1$ is a coalgebra map, $\Delta_1: C_1 \to C_1 \otimes C_1$ induces a well defined coproduct
$$\Delta_1: N_1(C) \to N_1(C) \otimes N_1(C)$$ for which we use exactly the same Sweedler type notation as in \ref{sweedler} above. 

Define a coproduct
$$\nabla^C_0: (\Omega N_*(C))_0 \to (\Omega N_*(C))_0 \otimes (\Omega N_*(C))_0$$ on the  degree zero elements of the cobar construction of $N_*(C)$ as follows.
On the unit $1 \in (\Omega N_*(C))_0$ we let $$\nabla^C_0(1) := 1_{R} \otimes 1_{R}$$ and on any $\{ x \} \in (\Omega N_*(C))_0$, where $x \in N_1(C)$, we define
$$
\nabla^C_0(\{ x \}):= \sum_{(x)} \{ \widetilde{x} \}  \otimes \{ \overline{x} \} + 1_{R} \otimes \{ x\} + \{x \} \otimes 1_{R}.
$$
Then extend $\nabla^C_0$ as an algebra map to monomials in  $(\Omega N_*(C))_0$ of arbitrary length. 

It is clear that $\nabla^C_0$ is natural, coassociative, cocommutative, and counital with counit given by the map
$$(\Omega N_*(C) )_0 = R \oplus T^{>0}s^{-1} (N_1(C)) \to R$$
which is the identity on the first summand and zero everywhere else. When $C$ is clear from the context we write $\nabla_0^C=\nabla_0.$

We now prove that the algebraic structure used to determine the fundamental group in Theorem \ref{adamscobar} is completely determined by the $\Omega$-quasi-isomorphism type of the simplicial cocommutative coalgebra of chains. 

\begin{proposition} \label{coproductonH_0}
Let $C$ be a connected simplicial cocommutative coalgebra. The coproduct $\nabla_0$ defined above induces a coalgebra structure on $H_0(\Omega N_*(C))$  making it into a cocommutative bialgebra. Furthermore, if $C=R X$ for a reduced Kan complex $X$, there is an isomorphism of bialgebras $H_0(\Omega N_*(C)) \cong R[\pi_1(X)]$.
\end{proposition}
\begin{proof}
We prove that $\nabla_0$ induces a coproduct on homology $H_0(\Omega N_*(C))$. Define a map
\[\nabla_1: (\Omega N_*(C))_1 \to (\Omega N_*(C))_0 \otimes  (\Omega N_*(C))_1 \bigoplus  (\Omega N_*(C))_1 \otimes  (\Omega N_*(C))_0 \] as follows.
On any $\{ y \} \in (\Omega N_*(C))_1$, where $y \in N_2(C)$, define
\begin{align*}
\nabla_1(\{y\}):= &\sum_{(\widetilde{y})} \sum_{(y)} \{d_2(\widetilde{\widetilde{y}}) |d_0(\overline{\widetilde{y}})\} \otimes \{\overline{y}\} +\{d_2(\widetilde{y})\} \otimes \{\overline{y}\}+ \{d_0(\widetilde{y})\} \otimes \{ \overline{y} \}+ \{ \widetilde{y}\} \otimes \{ d_1(\overline{y}) \} &\\
&+ \{ y\} \otimes 1_{R} + 1_{R} \otimes \{y\}.&
\end{align*}
Above we have denoted the coproduct $N_2(C) \to N_2(C)\otimes N_2(C)$ (induced by $\Delta_2: C_2 \to C_2 \otimes C_2$) by $y \mapsto \sum_{(y)} \widetilde{y} \otimes \overline{y}$. 

For any $\{y_1|...|y_n\} \in  (\Omega N_*(C))_1$, where each $y_i$ belongs to either $N_1(C)$ or $N(C_2)$, $\nabla_1$ is defined by extending multiplicatively, i.e. by letting
$$ \nabla_1(\{y_1|...|y_n\} ):=  \nabla_{|y_1|-1}(y_1) \cdot \nabla_{|y_2|-1}(y_2) \cdot ... \cdot \nabla_{|y_n|-1}(y_n),$$ 
where $|y_i|$ denotes the degree of $y_i$ and we have denoted by $\cdot$ the product on the tensor product of algebras $\Omega N_*(C) \otimes \Omega N_*(C).$ Then a straightforward computation yields that $$(D_{\Omega N_*(C)} \otimes id + id \otimes D_{\Omega N_*(C)}) \circ \nabla_1 = \nabla_0 \circ D_{\Omega N_*(C)},$$
where $D_{\Omega N_*(C)}$ is the cobar construction differential. This implies that $\nabla_0$ induces a well defined map on homology, which we denote by the same symbol
$$\nabla_0: H_0(\Omega N_*(C)) \to H_0(\Omega N_*(C)) \otimes H_0(\Omega N_*(C)).$$ By construction, this coproduct is a map of algebras. Hence, it defines an $R$-bialgebra structure on $H_0(\Omega N_*(C))$. In the case when $C=FX$ for a reduced Kan complex $X$ each $\Delta_n: C_n \to C_n \otimes C_n$ is group-like on the basis elements given by the simplices in $X_n$. Thus the coproduct $\nabla_0$ coincides with the corresponding one constructed in the proof of Theorem \ref{adamscobar}. 
\end{proof}

\begin{remark} The above coproduct $\nabla_0$ extends a construction of Baues described in the simply connected setting \cite{B98}. It may be also be interpreted as an explicit description of the coproduct on degree $0$ of the cobar construction of differential graded coalgebra over the surjection operad (a particular model for the $E_{\infty}$-operad).  
\end{remark}

\begin{definition} For any connected simplicial cocommutative coalgebra $C$ we call $H_0(\Omega N_*(C))$, equipped with the cocommutative bialgebra structure constructed above, the \textit{fundamental bialgebra of $C$}.
\end{definition}

\begin{remark} Theorem  \ref{adamscobar} says that for $X$ a reduced Kan complex,  $H_0(\Omega N_*(C))$ is isomorphic to  $R[\pi_1(X)]$ as a Hopf algebra, so that by applying the group-like elements functor $$\mathcal{G}: \textbf{Halg}_{R} \to \textbf{Grp}$$ from the category of Hopf algebras to the category of groups, we obtain a natural isomorphism of groups $$\mathcal{G}(H_0(\Omega N_*(C))) \cong \pi_1(X).$$ In the above isomorphism we use the fact that for any integral domain $R$ and group $G$ the set of group-like elements in $R[G]$ forms a group naturally isomorphic to $G$. Conceptually, Proposition \ref{coproductonH_0} is saying that the fundamental group of a Kan complex  is completely determined from the natural algebraic structure (simplicial cocommutative coalgebra) of the chains $C=RX$ and that if $C'$ is any other connected simplicial cocommutative coalgebra which is $\Omega$-quasi-isomorphic to $C$ then there is an isomorphism $\mathcal{G}(H_0(\Omega N_*(C'))) \cong \pi_1(X)$. Hence, the fundamental group is determined by the quadratic equation $$\nabla(\alpha)= \alpha \otimes \alpha,$$ for $\alpha \in H_0(\Omega N_*(C))$. 
\end{remark}

\subsection{The fundamental simplicial twisting cochain}\label{subsec:Twistingmorphismcomingfromthecobarconstruction} 
We can consider the cocommutative bialgebra $H_0(\Omega N_*(C))$  described above as a simplicial cocommutative bialgebra by placing  $H_0(\Omega N_*(C))$ in each degree and defining the face and degeneracy maps to be the identity maps. For notational simplicity we denote this simplicial cocommutative bialgebra by $\pi(C)$. The construction of $\pi(C)$ is natural with respect to maps of connected simplicial cocommutative coalgebras and consequently induces a functor
$$\pi: \textbf{scCoalg}_{R}^0 \to \textbf{scBialg}_{R},$$
 where $\textbf{scCoalg}_{R}^0$ denotes the category of connected simplicial cocommutative $R$-coalgebras and $\textbf{scBialg}_{R}$ the category of simplicial cocommutative (possibly non-commutative) $R$-bialgebras. We define a simplicial coalgebra twisting cochain from $C$ to $\pi(C)$ as follows.

\begin{proposition}
For any connected simplicial cocommutative coalgebra $C$ there is a simplicial coalgebra twisting cochain
\[
\tau:C_n \rightarrow \pi(C)_{n-1}
\]
for $n \geq 1$, defined by 
\[
\tau(x):=[\{d_2 ... d_n (x)\}]+\epsilon(x) e,
\]
where $[\{d_2 ... d_n (x)\}] \in H_0(\Omega N_*(C))$ denotes the homology class of $\{d_2...d_n(x)\} \in( \Omega N_*(C))_0$, $\epsilon(x)$ denotes the counit applied to $x$, and $e$ is the unit element of $\pi(C)$. 
\end{proposition}

\begin{proof}
We show that $\tau$ satisfies the equations of Definition \ref{def:MCequationcoalg} and that it is a map of coalgebras.

It is clear that equations $1$ and $3$ of Definition \ref{def:MCequationcoalg} are satisfied, so we check equations $2$ and $4$. For equation $2$ we need to check that
\[
\tau d_1 (x)= \mu (d_0 \tau \otimes \tau d_0) \Delta_{C} (x),
\]
where $\Delta_C$ is the coproduct of $C$. Using the simplicial identities we can write the left hand side as \[\tau d_1(x)= [\{d_1 d_3 .... d_n (x) \}]+\epsilon(x) e.\] On the other hand, the right hand side is equal to 
\[
\mu (d_0 \tau \otimes \tau d_0) \Delta_{C} (x)= \sum_{(x)} [\{d_2 ...d_n (\widetilde{x})|d_0 d_3 ... d_n(\overline{x})\}] + [\{d_0 d_3 ... d_n(x)\}] + [\{d_2 ... d_n(x)\}] +\epsilon(x) e.
\]
To show that these two are equal in homology we find an $\alpha \in (\Omega N_*(C))_1$ bounding
\begin{equation}\label{eq:boundary}
\sum_{(x)} \{d_2 ...d_n (\widetilde{x})|d_0 d_3 ... d_n(\overline{x}) \}+ \{d_0 d_3 ... d_n(x)\} + \{d_2 ... d_n(x)\} +\epsilon(x) e -(\{d_1 d_3 .... d_n (x)\}+\epsilon(x) e).
\end{equation}
One explicit choice of such a boundary is given by $\alpha= \{d_3 ... d_n(x)\}$; a straightforward computation shows that $D_{\Omega N_*(C)} \{d_3 ... d_n(x)\}$ is exactly equation \ref{eq:boundary}, proving that $\tau$ satisfies equation 2 of Definition \ref{def:MCequationcoalg}.

We now show that $\tau$ satisfies equation $4$ of Definition \ref{def:MCequationcoalg}, i.e. 
\[
(id_C \otimes \mu)\circ(id_C\otimes \tau s_0 \otimes id_{\pi(C)})\circ (\Delta_C \otimes id_{\pi(C)})(x\otimes  g)=x\otimes g,
\]
for any $x\otimes g \in C\otimes \pi(C)$. Note that 
\begin{align}
&(id_C \otimes \mu)\circ(id_C\otimes \tau s_0 \otimes id_{\pi(C)})\circ (\Delta_C \otimes id_{\pi(C)})(x\otimes  g)&=\\
& \sum_{(x)} \widetilde{x} \otimes \tau s_0(\overline{x}) g &= \\
 & \sum_{(x)} \widetilde{x} \otimes [\{d_2 ... d_{n+1}s_0(\overline{x})| g\}] +  \sum_{(x)} \widetilde{x} \epsilon(\overline{x}) \otimes g.&
\end{align}
 Using the simplicial identities we see that $d_2 ... d_{n+1}s_0(\bar{x})=s_0d_1...d_n(\bar{x})$, which is degenerate and therefore zero in the normalized chain complex $N_*(C)$. We are left with  $\sum_{(x)} \widetilde{x} \epsilon(\bar{x}) \otimes g$, but since $\epsilon$ is the counit and $e$ is the unit, this is exactly $x \otimes g$. 

We now prove that $\tau:C_1 \rightarrow \pi(C)_0$ is a coalgebra map. First note  $\tau:C_1 \rightarrow \pi(C)_0=H_0(\Omega N_*(C))$ is the sum of the projection map and the counit times the unit element. The projection map from $C_1$ to $H_0(\Omega N_*(C))$ is given by sending an element $x \in C_1$ to the homology class $[\{x\}] \in H_0(\Omega N_*(C)) $.

We verify  $$\nabla_0(\tau(x))=(\tau \otimes \tau)\Delta_1(x),$$ where $\Delta_1: C_1 \rightarrow C_1 \otimes C_1$ is the coproduct of $C_1$ and $\nabla_0:\pi(C) \rightarrow \pi(C) \otimes \pi(C)$ is the coproduct from Section \ref{sec:coalgebrastructureoncobar}. The right hand side is equal to
 \begin{align}
\sum_{(x)} \tau(\widetilde{x}) \otimes \tau(\overline{x})& = \sum_{(x)} [\{\widetilde{x}\}] \otimes [\{\overline{x}\}] +  [\{\widetilde{x}\}] \otimes \epsilon(\overline{x}) e + \epsilon(\widetilde{x}) e \otimes [\{\overline{x}\}]  + \epsilon(\widetilde{x}) e \otimes \epsilon(\overline{x}) e \\
&= \sum_{(x)} [\{\widetilde{x}\}] \otimes [\{\overline{x}\}] +  [\{x\}] \otimes  e + e \otimes [\{x\}]+ \epsilon(x) e \otimes  e,
\end{align}
where in the second line we used that $\epsilon$ is a counit. This is exactly the left hand side, since by the definition of $\nabla_0$ we have
\begin{align}
\nabla_0(\tau(x)) &=\nabla_0([x]) + \nabla_0(\epsilon(x)e) \\
&= \sum_{(x)} [\{\widetilde{x}\}] \otimes  [\{\overline{x}\}] +[\{x\}] \otimes e + e \otimes [\{x\}] + \epsilon(x)e \otimes e.
\end{align}
It now follows immediately that for each $n>1$ the map $\tau: C_n \to \pi(C)_{n-1}$ is a coalgebra map.
\end{proof}

\begin{definition} We call the map $\tau: C\to \pi(C)$, defined in the above proposition, the \textit{fundamental simplicial twisting cochain of $C$}.
\end{definition} 

\subsection{The universal cover}
We now assemble a new simplicial cocommutative coalgebra from a connected simplicial cocommutative coalgebra $C$ together with $\pi(C)$ and $\tau: C\to \pi(C).$

\begin{definition}
Let $C$ be a connected simplicial cocommutative coalgebra and $\tau:C \rightarrow \pi(C)$ the fundamental simplicial twisting cochain described in the previous section. The \textit{universal cover of C} is the simplicial cocommutative coalgebra defined as the simplicial twisted tensor product $C \otimes_\tau \pi(C)$ with respect to $\tau:C \rightarrow \pi(C)$ equipped with the coproduct described in Proposition \ref{coalgebraontwistedproduct}. Denote the universal cover of $C$ by $\widetilde{C}=C \otimes_\tau \pi(C)$. The universal cover is clearly a functorial construction with respect to maps of connected simplicial cocommutative coalgebras and so it defines a functor
$$\widetilde{ }: \textbf{scCoalg}^0_{R} \to \textbf{scCoalg}_{R}.$$ Note that $\widetilde{C}$ has a natural right $\pi(C)$-module structure. 
\end{definition}

When $C=RX$ for a reduced Kan complex $X$, then $\widetilde{C}$ is exactly the simplicial cocommutative coalgebra of chains on the universal cover $\widetilde{X}$. This is stated in the following proposition, which is a straightforward consequence of Corollary \ref{cor:universalcover}.

\begin{proposition}\label{universalcovercommutes}
If $X$ is a reduced Kan complex, then there is an isomorphism of simplicial cocommutative coalgebras between $R \widetilde{X}$ and $\widetilde{R X}$.
\end{proposition}

 The following lemma relates the simplicial twisted tensor product $C \otimes_{\tau} \pi(C)$ to Brown's twisted tensor product in the dg setting.
 
\begin{lemma}\label{Brown}
Let $C$, $\pi(C)$, and $\tau: C \to \pi(C)$ be as above. Then the dg $R$-module of normalized chains $N_*(C \otimes_{\tau} \pi(C))$ is naturally isomorphic to the twisted tensor product  $N_*(C) \otimes_{N(\tau)} N_*(\pi(C))$ in the sense of Brown.
\end{lemma}

\begin{proof}
The normalized chain complex $N_*(\pi(C))$ is a dg algebra concentrated in degree $0$. In fact, by definition $N_0(\pi(C)) =H_0(\Omega N_*(C))$.
Also note that, for any $n \geq 0$, we have $N_n(C \otimes \pi(C)) \cong N_n(C) \otimes H_0(\Omega N_*(C))$. Hence, we have a natural isomorphism of graded $R$-modules  $N_*(C \otimes_{\tau} \pi(C))\cong N_*(C) \otimes_{N(\tau)} N_*(\pi(C))$. To see that the differentials are the same, first observe that Brown's twisting cochain $N(\tau): N_n(C) \to N_{n-1}(\pi(C))$ is only non-zero for $n=1$ where it is given by the projection map $$N(\tau)(y) = [ \{y\} ] \in H_0(\Omega N_*(C)),$$ for any $y \in N_1(C)$. Then the twisting term in the differential of Brown's twisted tensor product is a map $d_{N(\tau)}: N_*(C) \otimes_{N(\tau)} H_0(\Omega N_*(C)) \to N_*(C) \otimes_{N(\tau)} H_0(\Omega N_*(C)) $ given by 
$$d_{N(\tau)}(x \otimes g) = \sum_{(x)} d_0(\widetilde{x}) \otimes g \cdot [\{d_2 ... d_n (\overline{x})\}]\ .$$
Thus, for any $x \otimes g \in N_n(C) \otimes_{N(\tau)} H_0(\Omega N_*(C))$, the differential of the Brown twisted tensor product is given by 
$$(d_C \otimes id + d_{N(\tau)})(x \otimes g)= \sum_{i=0}^n (-1)^id_i(x) \otimes g + \sum_{(x)} d_0(\widetilde{x}) \otimes g \cdot [\{d_2 ... d_n (\overline{x})\}],$$
which is precisely $\sum_{i=0}^n (-1)^i d_i^{\tau}(x \otimes g)$, the differential of $N_*(C \otimes_\tau \pi(C))$. 
\end{proof}
\begin{remark}
The above lemma does not hold in general if we replace $\pi(C)$ with an arbitrary simplicial algebra. If the non-degenerate simplicies of $A$ are not concentrated in degree $0$, there is no isomorphism $N_*(C \otimes A) \cong N_*(C) \otimes N_*(A)$; this can be solved using a cubical version of the twisted product, as described in \cite{KS05}. See  \cite{S61} for a more general statement regarding relation between the chains on a twisted Cartesian product and Brown's twisted tensor products. 
\end{remark}

Any $\Omega$-quasi-isomorphism of connected simplicial cocommutative coalgebras induces a weak equivalence between universal covers as we now show.

\begin{theorem}\label{universalcoverfunctor}
 The universal cover  functor $$\widetilde{ }: \textbf{scCoalg}^0_{R} \to \textbf{scCoalg}_{R}$$ sends $\Omega$-quasi-isomor\-phisms  between simplicial cocommutative $R$-flat coalgebras to  quasi-isomorphisms.
\end{theorem}

\begin{proof}
Let $f: C \to C'$ be an $\Omega$-quasi-isomorphism between two connected simplicial cocommutative $R$-flat coalgebras. Then $\Omega N_*(f): \Omega N_*(C) \to \Omega N_*(C')$ is a quasi-isomorphism of dg $R$-flat algebras. In particular, $H_0( \Omega N_*(f)): H_0(\Omega N_*(C)) \to H_0(\Omega N_*(C'))$ is an isomorphism of algebras (in fact, of bialgebras). By Theorem \ref{invarianceofbrown}, $f$ induces a quasi-isomorphism of chain complexes
$$N_*(f) \otimes H_0(\Omega N_*(f)): N_*(C) \otimes_{N(\tau)} H_0(\Omega N_*(C)) \to N_*(C') \otimes_{N(\tau)} H_0(\Omega N_*(C'))$$ 
between Brown twisted tensor products. Lemma \ref{Brown} then implies $$N_*(f \otimes \pi(f)): N_*(C \otimes_\tau \pi(C)) \to N_*(C' \otimes_\tau \pi(C'))$$ is a quasi-isomorphism.

\end{proof}

\section{Main theorem}

Using the machinery developed in the previous sections, together with Theorem C of \cite{G95}, we prove the following.

\begin{proposition} \label{algclosed}
Let $\mathbb{E}$ be an algebraically closed field and $X$ and $Y$ two reduced Kan complexes. If the connected simplicial cocommutative coalgebras of chains $\mathbb{E}X$ and $\mathbb{E}Y$ are $\Omega$-quasi-isomorphic then $X$ and $Y$ are  $\pi_1$-$\mathbb{E}$-equivalent.
\end{proposition}

\begin{remark}
In the proofs of Proposition \ref{algclosed} and Theorem \ref{anyfield}, we will need to use the naturality of the universal cover construction, it is therefore important that we use explicit zig-zags of equivalences and not just morphisms in the homotopy category. 
\end{remark}

\begin{proof}[Proof of Proposition \ref{algclosed}]
Suppose that $\mathbb{E}X$ and $\mathbb{E}Y$ are $\Omega$-quasi-isomorphic, so there is a zig-zag of connected simplicial cocommutative coalgebras
\begin{equation}\label{zigzag1}
\mathbb{E}X \xrightarrow{\simeq_{\Omega}}C_1 \xleftarrow{\simeq_{\Omega}} \cdots \xrightarrow{\simeq_{\Omega}} C_n \xleftarrow{\simeq_{\Omega}} \mathbb{E}Y,
\end{equation}
where each map is an $\Omega$-quasi-isomorphism. In particular, there are isomorphisms 
\begin{equation}\label{bialgebras}
H_0(\Omega N_*( \mathbb{E}X)) \cong H_0(\Omega N_*(C_1)) \cong \cdots \cong H_0(\Omega N_*(C_n)) \cong H_0(\Omega N_*( \mathbb{E}Y))
\end{equation}
 of fundamental bialgebras. Since $X$ and $Y$ are Kan complexes, Theorem \ref{adamscobar} implies that by applying the functor of group-like elements we obtain an isomorphism $$\pi_1(X)\cong \pi_1(Y) := \pi_1$$ of fundamental groups. Apply the universal cover functor \text{  }$\widetilde{ }: \textbf{scCoalg}^0_{\mathbb{E}} \to \textbf{scCoalg}_{\mathbb{E}}$ to the zig-zag in \ref{zigzag1} and obtain a zig-zag of simplicial cocommutative coalgebras
\begin{equation}\label{zigzag2}
\widetilde{\mathbb{E}X} \xrightarrow{\simeq} \widetilde{C_1} \xleftarrow{\simeq} \cdots \xrightarrow{\simeq} \widetilde{C_n} \xleftarrow{\simeq} \widetilde{\mathbb{E}Y}.
\end{equation}
Each object in the zig-zag \ref{zigzag2} has a $\pi_1$-action, each map is $\pi_1$-equivariant, and by Theorem \ref{universalcoverfunctor}, each map is a quasi-isomorphism of simplicial cocommutative coalgebras.  By Proposition \ref{universalcovercommutes}, the endpoints in zig-zag \ref{zigzag2} are naturally  $\pi_1$-equivariantly isomorphic to $\mathbb{E}\widetilde{X}$ and $\mathbb{E}\widetilde{Y}$, respectively. Thus, we get a zig-zag of $\pi_1$-equivariant quasi-isomorphisms of simplicial cocommutative $\pi_1$-coalgebras
\begin{equation}\label{zigzag3}
\mathbb{E}\widetilde{X} \xrightarrow{\simeq} \widetilde{C_1} \xleftarrow{\simeq} \cdots \xrightarrow{\simeq} \widetilde{C_n} \xleftarrow{\simeq} \mathbb{E}\widetilde{Y}.
\end{equation}
Apply to the above zig-zag of quasi-isomorphisms the composition of functors $(\mathcal{P} \circ \mathcal{R})$ where $\mathcal{R}: \textbf{sCoalg}_{\mathbb{E}} \to \textbf{sCoalg}_{\mathbb{E}}$ is a fibrant replacement functor for Goerss' model category structure on simplicial cocommutative coalgebras and $\mathcal{P}: \textbf{sCoalg}_{\mathbb{E}} \to \textbf{sSet}$ is the functor of points. We obtain a zig-zag of $\mathbb{E}$-local spaces with $\pi_1$-actions

\begin{equation}\label{zigzag4}
(\mathcal{P} \circ \mathcal{R})\mathbb{E}\widetilde{X} \xrightarrow{\simeq} (\mathcal{P} \circ \mathcal{R})\widetilde{C_1} \xleftarrow{\simeq} \cdots \xrightarrow{\simeq} (\mathcal{P} \circ \mathcal{R})\widetilde{C_n} \xleftarrow{\simeq}(\mathcal{P} \circ \mathcal{R}) \mathbb{E}\widetilde{Y}
\end{equation}
where each map is a $\pi_1$-equivariant weak homotopy equivalence. By Theorem C of \cite{G95} the endpoints of \ref{zigzag4} are the $\mathbb{E}$-localizations of the universal covers $\widetilde{X}$ and $\widetilde{Y}$, respectively. 

By adding the derived unit of adjunction ($\mathbb{E}, \mathcal{P})$ to each end of the above zig-zag, we obtain

\begin{equation}
\widetilde{X} \xrightarrow{\eta_{\widetilde{X}}} (\mathcal{P} \circ \mathcal{R})\mathbb{E}\widetilde{X} \xrightarrow{\simeq} (\mathcal{P} \circ \mathcal{R})\widetilde{C_1} \xleftarrow{\simeq} \cdots \xrightarrow{\simeq} (\mathcal{P} \circ \mathcal{R})\widetilde{C_n} \xleftarrow{\simeq}(\mathcal{P} \circ \mathcal{R}) \mathbb{E}\widetilde{Y}\xleftarrow{\eta_{\widetilde{Y}}} \widetilde{Y}.
\end{equation}

To turn this zig-zag of $\F$-equivalences between the $\pi_1$-equivariant Kan complexes $\tilde{X}$ and $\tilde{Y}$ into a zig-zag of $\pi_1$-$\F$-equivalences of Kan complexes between $X$ and $Y$, we would like to take the quotient by $\pi_1$ on every space in the zig-zag. We can not do this directly, because the derived functor of points does not necessarily produce simplicial sets with a free $\pi_1$-action. We therefore first need to apply the Borel construction to make sure that all the $\pi_1$ actions are free. After applying the Borel construction we get

\begin{equation}\label{lastzigzag} E \pi_1 \times_{\pi_1} \widetilde{X} \xrightarrow{\psi_X} E \pi_1 \times_{\pi_1} (\mathcal{P} \circ \mathcal{R})\mathbb{E}\widetilde{X}  \xrightarrow{\simeq}E\pi_1 \times_{\pi_1}  (\mathcal{P} \circ \mathcal{R})\widetilde{C_1} \xleftarrow{\simeq} \cdots 
\end{equation}
$$ \cdots \xrightarrow{\simeq} E\pi_1 \times_{\pi_1}  (\mathcal{P} \circ \mathcal{R})\widetilde{C_n} \xleftarrow{\simeq} E \pi_1 \times_{\pi_1} (\mathcal{P} \circ \mathcal{R})\mathbb{E}\widetilde{Y}  \xleftarrow{\psi_Y} E \pi_1 \times_{\pi_1} \widetilde{Y}.$$

We claim that every map in the sequence of spaces $\ref{lastzigzag}$ is a $\pi_1$-$\mathbb{E}$-equivalence. Note that because the two ends are weakly homotopy equivalent to $X$ and $Y$, respectively, this would establish a zig-zag of $\pi_1$-$\mathbb{E}$-equivalences between $X$ and $Y$. 

All of the following maps are weak homotopy equivalences and therefore $\pi_1$-$\mathbb{E}$-equivalences
\begin{equation} E \pi_1 \times_{\pi_1} (\mathcal{P} \circ \mathcal{R})\mathbb{E}\widetilde{X}  \xrightarrow{\simeq}E\pi_1 \times_{\pi_1}  (\mathcal{P} \circ \mathcal{R})\widetilde{C_1} \xleftarrow{\simeq} \cdots 
\end{equation}
$$ \cdots \xrightarrow{\simeq} E\pi_1 \times_{\pi_1}  (\mathcal{P} \circ \mathcal{R})\widetilde{C_n} \xleftarrow{\simeq} E \pi_1 \times_{\pi_1} (\mathcal{P} \circ \mathcal{R})\mathbb{E}\widetilde{Y}.$$

Therefore, we just need to argue that $\psi_X: E \pi_1 \times_{\pi_1} \widetilde{X} \to E \pi_1 \times_{\pi_1} (\mathcal{P} \circ R)\mathbb{E}\widetilde{X}$ is a $\pi_1$-$\mathbb{E}$-equivalence (the argument for $\psi_Y$ will be exactly the same). The map $\psi_X$ is the induced map on the coinvariants, forming the following commutative diagram
\[
\xymatrix{
\pi_1 \ar[r] \ar[d]_{=} & E\pi_1 \times \widetilde{X} \ar[r] \ar[d]^{id \times \eta_{\widetilde{X}}} &  E\pi_1 \times_{\pi_1} \widetilde{X} \ar[d]^{\psi_X} \\
 \pi_1 \ar[r] & E \pi_1 \times (\mathcal{P} \circ \mathcal{R})\mathbb{E}\widetilde{X}  \ar[r]  & E \pi_1 \times_{\pi_1} (\mathcal{P} \circ \mathcal{R})\mathbb{E}\widetilde{X},
}
\]
The two spaces in the middle column of the above diagram are simply connected and the horizontal maps in the right hand side square are universal covers. It follows that $\psi_X$ induces an isomorphism on fundamental groups.  The middle map is an $\mathbb{E}$-equivalence since $$\eta_{\widetilde{X}}: \widetilde{X} \to (\mathcal{P} \circ \mathcal{R}) \mathbb{E}\widetilde{X},$$ the derived unit of the adjunction $(\mathbb{E}, \mathcal{P})$,  is the Bousfield $\mathbb{E}$-localization of $\widetilde{X}$. Hence, $\psi_X$ is a $\pi_1$-$\mathbb{E}$-equivalence. 
\end{proof}

Recall that if $X$ is a simplicial set and $\mathbb{F}$ is \textit{any} field with algebraic closure $\mathbb{E}$ then the field extension $\mathbb{F} \hookrightarrow \mathbb{E}$ induces a weak homotopy equivalence between localizations $L_{\mathbb{F}}X \xrightarrow{\simeq} L_{\mathbb{E}}X$ \cite{G95}. Furthermore, for any field $\mathbb{F}$, the  simplicial cocommutative $\mathbb{F}$-coalgebra of chains on a space, under quasi-isomorphisms, determines the space up to Bousfield $\mathbb{F}$-localization (Theorem D in \cite{G95}). 
Using Proposition \ref{algclosed}, we show that, for \textit{any} field $\mathbb{F}$, the $\mathbb{F}$-chains in $X$, regarded as a connected simplicial cocommutative $\mathbb{F}$-coalgebra up to $\Omega$-quasi-isomorphism, determines reduced Kan complexes up to $\pi_1$-$\mathbb{F}$-equivalence. This is our main theorem. 

\begin{theorem} \label{anyfield} For any field $\mathbb{F}$, two reduced Kan complexes $X$ and $Y$ are  $\pi_1$-$\mathbb{F}$-equivalent if and only if the connected simplicial cocommutative coalgebras of chains $\mathbb{F}X$ and $\mathbb{F}Y$ are  $\Omega$-quasi-isomorphic. 
\end{theorem}
\begin{proof}

We first show that if $X$ and $Y$ are $\pi_1$-$\mathbb{F}$-equivalent then $\mathbb{F}X$ and $\mathbb{F}Y$ are $\Omega$-quasi-isomorphic. It suffices to prove that any map $f: X \to Y$ between reduced Kan complexes which induces an isomorphism on fundamental groups $$\pi_1(f): \pi_1(X)\cong \pi_1(Y) :=\pi_1$$ and an $\mathbb{F}$-equivalence between universal covers  $$\widetilde{f}: \widetilde{X} \to \widetilde{Y},$$ induces an $\Omega$-quasi-isomorphism $$N_*(f): N_*(\mathbb{F}X ) \to N_*(\mathbb{F}Y)$$ of connected dg coalgebras. 

Since $X$ and $Y$ are reduced, $\widetilde{X}$ and $\widetilde{Y}$ have induced base points. Let $\text{Sing}^1|\widetilde{X}|$ be the singular Kan complex consisting of all singular simplices $\sigma: |\Delta^n| \to |\widetilde{X}|$, for any $n\geq0$, that collapse the $1$-skeleton of $|\Delta^n|$ to the basepoint of $|\widetilde{X}|$ and define $\text{Sing}^1|\widetilde{Y}|$ similarly. Since $\widetilde{X}$ and $\widetilde{Y}$  are both simply connected, the inclusions of Kan complexes $$\text{Sing}^1|\widetilde{X}| \hookrightarrow \text{Sing}|\widetilde{X}|$$ and $$\text{Sing}^1|\widetilde{Y}| \hookrightarrow \text{Sing}|\widetilde{Y}|$$ are a natural homotopy equivalences. We claim that the map $$\text{Sing}^1|\widetilde{f}|: \text{Sing}^1|\widetilde{X}| \to \text{Sing}^1|\widetilde{Y}|$$
is an $\mathbb{F}$-equivalence, or equivalently, that it induces a quasi-isomorphism 
$$N_*(|\widetilde{f}|): N_*(\mathbb{F} \text{Sing}^1|\widetilde{X}|) \to N_*(\mathbb{F}\text{Sing}^1|\widetilde{Y}|)$$
between the simply connected dg coalgebras of normalized chains. This follows from the 2-out-of-3 axiom for $\mathbb{F}$-equivalences by considering the commutative diagram
\[
\xymatrix{
\text{Sing}^1|\widetilde{X}| \ar[r]^{\text{Sing}^1|\widetilde{f}|} \ar[d]^{\simeq} &\text{Sing}^1|\widetilde{Y}|  \ar[d]_{\simeq}\\
 \text{Sing}|\widetilde{X}|\ar[r]^{\text{Sing}|\widetilde{f}|} & \text{Sing}|\widetilde{X}|
 \\
\widetilde{X} \ar[u]^{\simeq} \ar[r]_{\widetilde{f}} & \widetilde{Y} \ar[u]_{\simeq}
}
\]
where the vertical maps are all homotopy equivalences, and consequently $\mathbb{F}$-equivalences, and the bottom horizontal map $\tilde{f}$ is an $\mathbb{F}$-equivalence by assumption.

By Proposition \ref{simpconnected}, quasi-isomorphisms of simply connected dg coalgebras are $\Omega$-quasi-isomorphisms, hence we get an induced quasi-isomorphism
$$\Omega N_*(|\widetilde{f}|): \Omega N_*(\mathbb{F} \text{Sing}^1|\widetilde{X}|) \to  \Omega N_*(\mathbb{F} \text{Sing}^1|\widetilde{Y}|) $$
of dg algebras. By Adams' classical cobar theorem, it follows that $\widetilde{f}: \widetilde{X} \to \widetilde{Y}$ induces a quasi-isomorphism
$$ C_*(\Omega |\widetilde{f}|): C_*(\Omega | \widetilde{X}|;\mathbb{F} ) \to C_*(\Omega |\widetilde{Y}|; \mathbb{F} )$$
between the dg algebras of normalized  singular chains on the based loop spaces of $\widetilde{X}$ and $\widetilde{Y}$, respectively. It now follows that the dg algebra map $$C_*(\Omega |f|): C_*(\Omega |X|; \mathbb{F}) \to C_*(\Omega| Y|; \mathbb{F})$$ is a quasi-isomorphism since we have a commutative square

\[
\xymatrix{
 H_*(\Omega |\widetilde{X}|;\mathbb{F}) \otimes \mathbb{F}[\pi_1] \ar[d]_{H_*(\Omega |\widetilde{f}|) \otimes id} \ar[r] &H_*(\Omega |X|; \mathbb{F}) \ar[d]^{H_*(\Omega |f|)}\\
 H_*(\Omega| \widetilde{Y}|;\mathbb{F}) \otimes \mathbb{F}[\pi_1] \ar[r] &H_*(\Omega |Y|; \mathbb{F})
}
\]
where the left vertical map and the horizontal maps, which are induced by projecting from the universal cover to the base and then multiplying of loops, are isomorphisms of graded vector spaces. By the extension of Adams' cobar theorem to reduced Kan complexes, recalled in Theorem \ref{nscadams}, it follows that $N_*(f)$ is an $\Omega$-quasi-isomorphism. 

To prove the converse suppose the connected simplicial cocommutative coalgebras of chains $\mathbb{F}X$ and $\mathbb{F}Y$ are $\Omega$-quasi-isomorphic through a zig-zag 
\begin{equation}\label{zigzag6}
\mathbb{F}X \xrightarrow{\simeq_{\Omega}}C_1 \xleftarrow{\simeq_{\Omega}} \cdots \xrightarrow{\simeq_{\Omega}} C_n \xleftarrow{\simeq_{\Omega}} \mathbb{F}Y.
\end{equation}
Let $\mathbb{E}$ be the algebraic closure of $\mathbb{F}$ and tensor zig-zag \ref{zigzag6} with $\mathbb{E}$ to obtain a zig-zag of simplicial $\mathbb{E}$-coalgebras
\begin{equation}
\mathbb{E}X \cong \mathbb{F}X \otimes_{\mathbb{F}} \mathbb{E} \xrightarrow{\Omega}C_1\otimes_{\mathbb{F}} \mathbb{E} \xleftarrow{\Omega} \cdots \xrightarrow{\Omega} C_n \otimes_{\mathbb{F}} \mathbb{E}  \xleftarrow{\Omega}\mathbb{F}X \otimes_{\mathbb{F}} \mathbb{E} \cong \mathbb{E}Y.
\end{equation}
All the maps above are $\Omega$-quasi-isomorphisms because for any dg $\mathbb{F}$-coalgebra $C$, $\Omega(C \otimes_{\mathbb{F}} \mathbb{E})\cong\Omega(C) \otimes_{\mathbb{F}} \mathbb{E}$ and tensoring over a field preserves quasi-isomorphisms. It follows from Proposition \ref{algclosed} that $X$ and $Y$ are $\pi_1$-$\mathbb{E}$-equivalent, which implies $X$ and $Y$ are $\pi_1$-$\mathbb{F}$-equivalent.
\end{proof}

We now consider a reformulation of Theorem \ref{anyfield} in the case $\mathbb{F}=\mathbb{Q}$. A map $f: X \to Y$ between reduced Kan complexes is called a $\pi_1$-\textit{rational homotopy equivalence} if it induces an isomorphism on fundamental groups $$\pi_1(f): \pi_1(X) \xrightarrow{\cong} \pi_1(Y)$$ and an isomorphism between rationalized higher homotopy groups $$\pi_n(f) \otimes \mathbb{Q} : \pi_n(X) \otimes \mathbb{Q} \xrightarrow{\cong} \pi_n(Y) \otimes \mathbb{Q}$$ for $n\geq 2$ \cite{RWZ19}. The following corollary extends the classification theorem of rational homotopy theory to path-connected spaces with arbitrary fundamental group. 

\begin{corollary} Two reduced Kan complexes $X$ and $Y$ are $\pi_1$-rationally homotopy equivalent if and only if the connected simplicial cocommutative coalgebras of chains $\mathbb{Q}X$ and $\mathbb{Q}$Y are $\Omega$-quasi-isomorphic. 
\end{corollary}
\begin{proof} This follows directly from Theorem \ref{anyfield} because the notion of $\pi_1$-$\mathbb{Q}$-equivalence is equivalent to that of $\pi_1$-rational equivalence, since $\mathbb{Q}$-equivalences between simply connected spaces (such as the universal covers) are exactly maps that induce isomorphisms on rationalized homotopy groups. 
\end{proof}

We would like to point out that in this corollary we do not need any finiteness assumptions on our spaces and that the simplicial cocommutative coalgebra of chains is therefore a complete invariant of the  $\pi_1$-rational homotopy type of reduced Kan complexes. The $\pi_1$-rational homotopy type was also studied in \cite{GHT00}, where they constructed a version of the minimal models for certain finite type spaces. Their results unfortunately had rather strong finiteness assumptions that for example excluded simple spaces like $S^1 \vee S^3$.

\section{The integral case for spaces with universal cover of finite type}

In this section we apply the algebraic machinery developed in sections 4 and 5 to prove a partial case of the conjecture posed in the introduction. First, one of the directions of the conjecture holds in complete generality.

\begin{theorem} If $X$ and $Y$ are homotopy equivalent reduced Kan complexes then the connected simplicial cocommutative coalgebras of integral chains $\mathbb{Z}X$ and $\mathbb{Z}Y$ are $\Omega$-quasi-isomorphic. 
\end{theorem}

\begin{proof} The proof for this fact follows by exactly the same argument as in the proof of the forward direction of Theorem \ref{anyfield}, which holds for any integral domain $R$. Note that a map is a $\pi_1$-$\mathbb{Z}$-equivalence if and only if it is a weak homotopy equivalence. 
\end{proof}

A simplicial $R$-coalgebra $C$ is said to be $R$-\textit{projective} if each $C_n$ is a projective $R$-module.  Note that when $R=\mathbb{Z}$ this is equivalent to a simplicial coalgebra $C$ such that each $C_n$ is  free as an abelian group. We further say that that the simplicial $R$-coalgebra $C$ is of finite type if $H_*(C)$ is finitely generated as an $R$-module in each degree. The combination of projective and finite type implies that we can dualize and use cochains instead of chains without losing any homotopical information. We prove a special case of the converse direction of our conjecture by applying our algebraic methods together with the main theorem of \cite{M06}, which says that the integral $E_{\infty}$-algebra of singular cochains determines finite type nilpotent homotopy types, at the level of universal covers. 

\begin{theorem} Let $X$ and $Y$ be two reduced Kan complexes whose universal covers are of finite type. If $\mathbb{Z}$X and $\mathbb{Z}Y$ can be connected by a zig-zag of $\Omega$-quasi-isomorphisms of connected simplicial cocommutative $\mathbb{Z}$-projective coalgebras, then $X$ and $Y$ are homotopy equivalent.  
\end{theorem}

\begin{proof}
Suppose that $\mathbb{Z}$X and $\mathbb{Z}Y$ can be connected by a zig-zag of $\Omega$-quasi-isomorphisms of connected simplicial cocommutative $\mathbb{Z}$-projective coalgebras
$$\mathbb{Z}X \xrightarrow{\simeq_{\Omega}}C_1 \xleftarrow{\simeq_{\Omega}} \cdots \xrightarrow{\simeq_{\Omega}} C_n \xleftarrow{\simeq_{\Omega}} \mathbb{Z}Y.$$
By Theorem \ref{adamscobar}, $\pi_1(X) \cong \pi_1(Y):= \pi_1$. By Theorems \ref{adamscobar} and \ref{universalcoverfunctor} we obtain a zig-zag of simplicial cocommutative coalgebras
$$\mathbb{Z}\widetilde{X}  \xrightarrow{\simeq} \widetilde{C_1} \xleftarrow{\simeq} \cdots \xrightarrow{\simeq} \widetilde{C_n} \xleftarrow{\simeq}  \mathbb{Z} \widetilde{Y}$$
where each object is a simplicial cocommutative $\mathbb{Z}$-projective coalgebra equipped with a natural $\pi_1$-action and each map is a $\pi_1$-equivariant quasi-isomorphism.

If $C$ is any simplicial cocommutative coalgebra, then the dg coassociative coalgebra of normalized chains $N_*(C)$ has a natural $E_{\infty}$-coalgebra struture through the construction described in \cite{BF04}. We remark here that strictly speaking Mandell's results require a cofibrant $E_\infty$-operad, while the Barratt-Eccles operad which Berger and Fresse use is just a $\Sigma$-cofibrant operad, this can be fixed by taking a cofibrant replacement. If $C=\mathbb{Z}S$ for some simplicial set $S$, this recovers the $E_{\infty}$-coalgebra of normalized chains on the simplicial set $S$. Moreover, if $C$ is $\mathbb{Z}$-projective, so is $N_*(C)$.  Applying the normalized chains functor to the above zig-zag we obtain a zig-zag of $\pi_1$-equivariant quasi-isomorphisms between $\mathbb{Z}$-projective $E_{\infty}$-coalgebras
$$N_*(\mathbb{Z}\widetilde{X} )  \xrightarrow{\simeq} N_*( \widetilde{C_1}) \xleftarrow{\simeq} \cdots \xrightarrow{\simeq} N_*(\widetilde{C_n} ) \xleftarrow{\simeq} N_*( \mathbb{Z} \widetilde{Y}).$$
Since each object above is $\mathbb{Z}$-projective, if we take linear duals (i.e. apply $\text{Hom}_{\mathbb{Z}}( \_ , \mathbb{Z})$) we obtain a zig-zag of $\pi_1$-equivariant quasi-isomorphisms of $E_{\infty}$-algebras

$$N^*(\mathbb{Z}\widetilde{X} )  \xleftarrow{\simeq} N^*( \widetilde{C_1}) \xrightarrow{\simeq} \cdots \xleftarrow{\simeq} N^*(\widetilde{C_n} ) \xrightarrow{\simeq} N^*( \mathbb{Z} \widetilde{Y}).$$

Since $\widetilde{X}$ and $\widetilde{Y}$ are of finite type by assumption, the main theorem in \cite{M06} implies that $\widetilde{X}$ and $\widetilde{Y}$ are homotopy equivalent. Furthermore, we may conclude $\widetilde{X}$ and $\widetilde{Y}$ are $\pi_1$-equivariantly homotopy equivalent. In fact, Mandell proves the main theorem in \cite{M06} by constructing a map  $$\epsilon: [N^*(\mathbb{Z}\widetilde{Y}), N^*(\mathbb{Z}\widetilde{X})]_{E_{\infty}} \to [\widetilde{X}, \widetilde{Y}]$$
which is natural at the level of homotopy categories, where $[ \_ , \_ ]_{E_{\infty}}$ and $[ \_, \_]$ denote the Hom sets  in the homotopy category of $E_{\infty}$-algebras and spaces, respectively, see Theorem 0.1 of \cite{M06}. By fixing a functorial fibrant and cofibrant replacements,  $\epsilon$ may be constructed so that it associates functorially a zig-zag of maps of spaces to any zig-zag of maps between $E_{\infty}$-algebras before passing to the homotopy category. Hence, this model for $\epsilon$ sends a zig-zag of $\pi_1$-equivariant maps of $E_{\infty}$-algebras to a zig-zag of $\pi_1$-equivariant maps between spaces. The desired conclusion now follows by taking Borel constructions as in the proof of Proposition \ref{algclosed}.

\end{proof}

\begin{remark}
It seems plausible that if $E_{\infty}$-algebras are replaced by flat $E_{\infty}$-coalgebras the finite type assumption in Mandell’s theorem might be dropped. Such a version of Mandell’s theorem could be used to prove the conjecture from the introduction. Another approach could be to prove an integral version of Theorem 10 (Theorem D in [G95]). Proving such a theorem is beyond the scope of this paper and will be the topic of future work.
\end{remark}

\bibliographystyle{plain}


\end{document}